\documentclass{article}
\usepackage[utf8]{inputenc}

\usepackage[all,2cell]{xy}

\usepackage{graphicx}
\usepackage{amsmath}
\usepackage{amsfonts}
\usepackage{amssymb}
\usepackage{appendix}
\usepackage{amsthm}
\newcommand{\Dial}{\mathfrak{Dial}(\mathsf{C})}
\newcommand{\mC}{\mathcal{C}}
\newcommand{\mV}{\mathcal{V}}
\newcommand{\mD}{\mathcal{D}}

\newcommand{\fibration}[3]{#2 \colon #1 \longrightarrow #3}

\newcommand{\arrow}[3]{#1 \xrightarrow{#2} #3}

\newcommand{\injarrow}[3]{\xymatrix{#1 \ar@{^{(}->}[r]^{#2} &#3}}

\def\op{\operatorname{op}}

\def\x{\times}
\def\pr{\pi}

\newcommand{\twomorphism}[6]{\xymatrix{
#1^{\op} \ar[rrd]^#2_{}="a" \ar[dd]_{#3^{\op}}\\
&& \pos\\
#5^{\op}  \ar[rru]_#6^{}="b"
\ar_{#4}  "a";"b"}}
\newcommand{\comsquare}[8]{ \xymatrix@+1pc{ 
#1 \ar[r]^{#5} \ar[d]_{#6} & #2 \ar[d]^{#7} \\
#3 \ar[r]_{#8} & #4 
}}

\newcommand{\dial}[1]{\mathfrak{Dial}(#1)}

\newcommand{\bemph}[1]{\textbf{\emph{#1}}}
\newcommand{\angbr}[2]{\langle #1,#2 \rangle}

\def\mG{\mathcal{G}}
\def\mX{\mathcal{X}}
\newcommand{\Pred}[1]{\mathbf{Pred}(#1)}

\def\pos{\operatorname{\mathbf{Pos}}}

\def\hey{\operatorname{\mathbf{Hey}}}

\def\set{\operatorname{\mathbf{Set}}}

\newcommand{\quadratocomm}[8]{ \xymatrix@+1pc{ 
#1 \ar[r]^{#5} \ar[d]_{#6} & #2 \ar[d]^{#7} \\
#3 \ar[r]_{#8} & #4 
}}

\renewcommand{\prod}{\forall}
\newcommand{\doctrine}[2]{#2 \colon #1^{\op}  \longrightarrow \pos}

\newcommand{\hyperdoctrine}[2]{#2 \colon #1^{\op}  \longrightarrow \hey}
\newcommand{\exfreedoctrine}[1]{#1^{\exists\text{-free}}}
\newcommand{\unfreedoctrine}[1]{#1^{\forall\text{-free}}}
\newcommand{\excomp}[1]{#1^{\exists}}
\newcommand{\uncomp}[1]{#1^{\forall}}
\newcommand{\quantfreedoctrine}[1]{#1^{\exists\forall\text{-free}}}
\newcommand{\CatDoc}{\mathsf{Doc}}
\newcommand{\ED}{\mathsf{ExD}}
\newcommand{\UD}{\mathsf{UnD}}

\usepackage{hyperref}

 \theoremstyle{plain} 
 \newtheorem{theorem}{Theorem}[section]
 \newtheorem{corollary}[theorem]{Corollary}
 
 \newtheorem{proposition}[theorem]{Proposition}

 \theoremstyle{definition} 
 \newtheorem{definition}[theorem]{Definition}
 \newtheorem{remark}[theorem]{Remark}

\title{Dialectica Principles via Gödel Doctrines\thanks{Research supported by the project MIUR PRIN 2017FTXR IT-MaTTerS (Trotta) and by a School of Mathematics EPSRC Doctoral Studentship (Spadetto)}}
\author{Davide Trotta \\ \footnotesize University of Pisa \\ \scriptsize \textsf{trottadavide92@gmail.com} \and Matteo Spadetto \\ \footnotesize University of Leeds \\ \scriptsize \textsf{matteo.spadetto.42@gmail.com} \and Valeria de Paiva \\ \footnotesize Topos Institute \\ \scriptsize \textsf{valeria@topos.institute}}
\date{}
\begin{document}

\maketitle

\begin{abstract}
Gödel's Dialectica interpretation was conceived as a tool to obtain the consistency of Peano arithmetic via a proof of consistency of Heyting arithmetic in the 40s. In recent years, several proof theoretic transformations, based on Gödel's Dialectica interpretation, have been used systematically to extract new content from classical proofs, following a suggestion of Kreisel. Thus, the interpretation has found new relevant applications in several areas of mathematics and computer science. 
Several authors have  explained the Dialectica interpretation in categorical terms. In particular, de Paiva introduced the notion of a \emph{Dialectica category} as an internal version of Gödel's Dialectica Interpretation in her doctoral work. This was generalised by Hyland and Hofstra, who considered the interpretation in terms of fibrations. In our previous work,  we introduced an \emph{intrinsic} presentation of the Dialectica construction via a generalisation of Hofstra's work, using the notion of 
Gödel fibration and its proof-irrelevant version, a Gödel doctrine. 
The key idea is that G\"odel fibrations (and doctrines) can be thought of as fibrations generated by some basic elements playing the role of \emph{quantifier-free elements}. 
This categorification of quantifier-free elements is crucial not only to show that our notion of  Gödel fibration
is equivalent to Hofstra's Dialectica fibration in the appropriate way, but also to show how Gödel doctrines embody the main logical features of the Dialectica Interpretation. 
To show that, we derive the soundness of the interpretation of the implication connective, as expounded by Troelstra, in the categorical model. This requires extra logical principles, going beyond intuitionistic logic, namely Markov Principle (MP) and the Independence of Premise (IP) principle, as well as some choice. We show how these principles are satisfied in the categorical setting, establishing a tight (internal language) correspondence between the logical system and the categorical framework. This tight correspondence should come handy not only when
discussing the traditional applications of the Dialectica, but also when dealing with some newer uses of the 
interpretation, as in  modelling games or  concurrency theory. 
Finally, to complete our analysis, we characterise categories obtained as results of the tripos-to-topos of Hyland, Johnstone and Pitts applied to Gödel doctrines.

\bigskip


\end{abstract}

\section{Introduction}
Gödel is known, amongst many other things, by his Platonistic views of Mathematics. This is not what the  `Gödel doctrines' in our title are about. Doctrines in our title are meant in Lawvere's  sense of `hyperdoctrines', a special kind of functor in Category Theory~\cite{lawvere1970}. Gödel doctrines then refer to Gödel because of his Dialectica Interpretation \cite{goedel1986}, a modified form of Hilbert's program, which shows consistency of logical systems by interpreting these systems into collections of ``computable functionals". First and most importantly, by showing consistency of Heyting arithmetic using his system $T$ of functionals. 


 In recent years, proof theoretic transformations (also called proof interpretations) based on Gödel's Dialectica interpretation~\cite{goedel1986} have been used systematically  by Kohlenbach and others~\cite{Kohlenbach2008} to extract new content from proofs, following Kreisel's suggestion. Thus, the Dialectica interpretation has found relevant applications in several areas of mathematics and computer science. 
 Meanwhile, several authors have  explained the Dialectica interpretation in categorical terms. In particular, de Paiva~\cite{dePaiva1989dialectica} introduced the notion of a \emph{Dialectica category} as an internal version of Gödel's Dialectica Interpretation. The idea is to construct a category $\Dial(\mathsf{C})$ from a category $\mathsf{C}$ with finite limits.  The main focus in de Paiva's original work is on the categorical structure of the category obtained, as this notion 
 of  Dialectica category turns out to be also a model of Girard's Linear Logic~\cite{girard1987}.

This construction was first generalised by Hyland, who investigated the Dialectica construction associated to a fibred preorder \cite{Hyland2002}. Later,  Biering in her PhD work \cite{Biering_dialecticainterpretations}  studied the Dialectica construction for an arbitrary cloven fibration.  
 Hofstra~\cite{hofstra2011} then wrote an exposition and interpretation of the Dialectica construction from a modern categorical perspective, emphasising its universal properties. His work gives  centre stage to the well-known concepts of pseudo-monads, simple products and co-products.

Taking advantage of the abstract presentation of Hofstra, in previous works \cite{trotta_et_al:LIPIcs.MFCS.2021.87,trotta-lfcs2022} we introduced an \emph{intrinsic} presentation of the Dialectica construction via the notion of Gödel fibration and its proof-irrelevant version, namely Gödel doctrines. 
The key idea is that Gödel fibrations (and doctrines) can be thought of as fibrations generated by some basic elements playing the role of \emph{quantifier-free elements}. 
This categorification of quantifier-free elements is crucial not only to show that the notions of Gödel fibrations introduced in \cite{trotta_et_al:LIPIcs.MFCS.2021.87} and Dialectica fibrations (as presented in \cite{hofstra2011}) are mathematically equivalent in the appropriate way, but also to show how Gödel doctrines embody the main logical features of the Dialectica Interpretation \cite{trotta-lfcs2022,dialecticaprinciples2022}.
While in \cite{trotta_et_al:LIPIcs.MFCS.2021.87} we presented a reconstruction of where the categorification of concepts came from, in \cite{dialecticaprinciples2022} we showed that this categorification worked not only for rules as in \cite{trotta-lfcs2022}, but also for the logical principles themselves, which is always more exciting for logicians than for category theorists.

The main purpose of the work here is to provide a self-contained and complete study of Gödel doctrines, presenting in detail the results developed in \cite{trotta-lfcs2022,dialecticaprinciples2022}, their connections with \cite{trotta_et_al:LIPIcs.MFCS.2021.87}, and carrying on the analysis of these doctrines. 
In this paper, we  present in detail the notion of \emph{doctrine} and its logical meaning first, and then we discuss the Dialectica interpretation of the implicational connective and the logical principles involved in the categorical interpretation of this connective. Moreover, in order to present our results with a notation familiar to both logicians and category theorists,  we will employ the \emph{internal language} of a doctrine \cite{pitts95}.

As far as the notion of doctrine is concerned, we follow the notation and the definitions presented in \cite{maiettirosolini13a,maiettirosolini13b}, where the authors introduce primary, elementary and existential doctrines as generalisations of the original notion of Lawvere's hyperdoctrine~\cite{lawvere1970}. 

After recalling the main categorical tools involved, we are going to present a (hyper)doctrine characterisation of the Dialectica interpretation which corresponds exactly to its logical description. The \emph{soundness} of the interpretation of the implication connective, as expounded on by Spector and Troelstra \cite{Troelstra1973}, in the categorical models will follow as a direct consequence of this tight correspondence. In particular, recall that such an interpretation is motivated by the equivalence: 
 \begin{center} $(\exists u.\forall x . \psi_D (u,x)\to\exists v.\forall y.  \phi_D (v,y)) \leftrightarrow$
 
 $\leftrightarrow\exists f_0,f_1.\forall u,y. (\psi_D(u,f_1(u,y))\rightarrow \phi_D (f_0(u),y))$
\end{center} where $\psi_D$ and $\phi_D$ are quantifier-free formulae.  Showing this equivalence requires extra logical principles, going beyond intuitionistic logic, specifically Markov Principle (MP) and the Independence of Premise (IP) principle, as well as some choice.  
 While the traditional categorical approach takes this equivalence as the starting point for defining categorical models, for example Dialectica categories \cite{dePaiva1989dialectica}, our approach focuses instead on abstracting in the setting of doctrines the key logical features that allow us to conclude such an equivalence.
 
 We show how these key logical features are satisfied in the categorical setting, establishing a tight correspondence between the logical system and the categorical framework. Our results are built on the categorical presentation of existential and universal-free elements we introduced first in the context of fibrations in \cite{trotta_et_al:LIPIcs.MFCS.2021.87}, and then in the language of doctrines~\cite{trotta-lfcs2022}.
 %
 Having a categorification of such notions is fundamental to properly state logical principles in categorical terms, since both (IP) and (MP) involve quantifier-free formulae.
 Finally, to complete our analysis of Gödel (hyper)doctrines, we characterise categories obtained as results of the tripos-to-topos  construction of Hyland, Johnstone and Pitts~\cite{hyland89} applied to these doctrines. After recalling the notions and the construction of the category of predicates associated to a hyperdoctrine from \cite{maiettirosolini13a},  and of exact completion of a lex category \cite{CARBONI199879}, we present an explicit characterisation of the tripos-to-topos construction associated to a Gödel hyperdoctrine. In particular, combining our results with the characterisation of exact completions presented in \cite{trottamaietti2020}, we show that every category obtained as tripos-to-topos of a Gödel hyperdoctrine can be equivalently presented as the exact completion of the category of predicates associated to the hyperdoctrine itself. We then conclude with a very brief discussion of other work we envisaged following from our characterisation.


\section{Doctrines}
One of the most relevant notions of categorical logic which enabled the study of logic from a pure algebraic perspective is that of a \emph{hyperdoctrine},  introduced in a series of seminal papers by F.W. Lawvere to synthesise the structural properties of logical systems \cite{lawvere1969,lawvere1969b,lawvere1970}.
Lawvere’s crucial intuition was to consider logical languages and theories as fibrations to study their 2-categorical properties, so that e.g. connectives, quantifiers and equality are determined by structural adjunctions. If theories and models can be viewed as objects and morphisms of a suitable category, i.e. the category of hyperdoctrines, this is in particular a 2-category, where 2-cells represent morphisms of models. So, having a 2-categorical structure allows us not only to compare theories (objects) via models (1-cells), but also to compare models (1-cells) via the 2-cells that represent morphisms of models.



Recall from \cite{pitts95} that a \textbf{first-order hyperdoctrine} is a contravariant functor:
\[\hyperdoctrine{\mC}{P}\]
from a cartesian category $\mathcal{C}$ to the category of Heyting algebras ${\mathbf{Hey}}$ satisfying:
for every arrow $\arrow{A}{f}{B}$ in $\mathcal{C}$, the homomorphism $\fibration{P(B)}{P_f}{P(A)}$ of Heyting algebras, where $P_f$ denotes the action of the functor $P$ on the arrow $f$, has a left adjoint $\exists_f$ and a right adjoint $\forall_f$.
These adjoints satisfy the Beck-Chevalley conditions (BC), i.e. for any pullback square:
\[\xymatrix@+1pc{
D\ar[d]_h\ar[r]^k & C \ar[d]^g\\
A \ar[r]_f & B
}\]
it is the case that the squares: \begin{center}$\xymatrix@+1pc{
PD\ar[r]^{\exists_k} & PC \\
PA\ar[u]^{P_h} \ar[r]_{\exists_f} & PB \ar[u]_{P_g}
}$ \,\, and \,\, $\xymatrix@+1pc{
PD\ar[r]^{\forall_k} & PC \\
PA\ar[u]^{P_h} \ar[r]_{\forall_f} & PB \ar[u]_{P_g}
}$\end{center} commute, i.e. the equalities:

\begin{center}$\exists_kP_h=P_g \exists_f$ \,\, and \,\, $\forall_kP_h=P_g \forall_f$\end{center} hold. 

Observe that in fact the pointwise inequalities $\exists_{k}P_{h}\leq P_g \exists_{f}$ and $\forall_{k}P_{h}\geq P_g \forall_{k}$ always hold from the adjunctions definitions. The purpose of the Beck-Chevalley conditions is to guarantee that substitution  commutes with quantification, appropriately, thus BC forces the equality in both diagrams.


A first-order hyperdoctrine determines an appropriate categorical structure to abstract a first-order theory and its corresponding Tarski semantics. Semantically, a first-order hyperdoctrine is essentially a generalisation of the contravariant \emph{powerset functor} on the category of sets: $$\hyperdoctrine{\set}{\mathcal{P}}$$ 
sending a set $A$ into the Heyting algebra $\mathcal{P}(A)$ of its subsets (ordered by  inclusion), and a set-theoretic function $\arrow{A}{f}{B}$ to the inverse image functor $\arrow{\mathcal{P}B}{\mathcal{P}f=f^{-1}}{\mathcal{P}A}$. 
In this case, the adjoints $\forall_f$ and $\exists_f $ must be evaluated, on a subset $D$ of $A$, respectively as the subsets 
$ \exists_f(D)=\{ a\in B\; | \; \exists a\in A. \; (b=f(a) \wedge a\in D) \}$ and $ \forall_f(D)=\{ a\in B\; | \; \forall a\in A. \; (b=f(a) \Rightarrow a\in D) \}$.

From a syntactic point of view, a first-order hyperdoctrine can be seen as a generalisation of the \emph{Lindenbaum-Tarski algebra} of well-formed formulae of a first order theory. In particular, given a first-order theory $\textsc{T}$ in a many-sorted first-order language $\mathcal{L}$, one can consider the functor:
$$\hyperdoctrine{\mV}{\mathcal{L}\textsc{T}}$$ 
whose base category $\mathcal{V}$ is the \emph{syntactic} category of $\mathcal{L}$, i.e. the one whose objects are  ($\alpha$-equivalence classes of) finite lists $\overrightarrow{x}:=[x_1:X_1,\dots,x_n:X_n]$ of typed variables and whose morphisms are lists of substitutions, while the elements of  $\mathcal{L}\textsc{T}(\overrightarrow{x})$ are given by equivalence classes (with respect to provable reciprocal
consequence $\dashv\vdash $) of well-formed formulae in the context $\overrightarrow{x}$, and order is given by the provable consequences, according to the fixed theory $\textsc{T}$. In this case, the left adjoint to the weakening functor $\mathcal{L}\textsc{T}_{\pr}$ is computed by existentially quantifying the variables that are not involved in the substitution induced by the projection (dually the right adjoint is computed by quantifying universally).

\subsection{Existential and universal doctrines}
Recently, several generalisations of the notion of a Lawvere hyperdoctrine were considered, and we refer for example to \cite{maiettipasqualirosolini,maiettirosolini13a,maiettirosolini13b} or to  \cite{pitts02,hyland89} for higher-order versions. 
For further insights and applications to higher-order logic or realisability, we refer to \cite{hyland89,van_Oosten_realizability,pitts02}.

In this work we consider a natural generalisation of the notion of first-order hyperdoctrine, and we call it simply a \emph{doctrine}.
\begin{definition}
A \textbf{doctrine} is a contravariant functor:
$$\doctrine{\mC}{P}$$ 
where the category $\mathcal{C}$ has finite products and $\pos$ is the category of posets.
\end{definition}

Now we recall from \cite{maiettipasqualirosolini,maiettirosolini13a,trotta2020} the notions of existential and universal doctrines.

\begin{definition}\label{def existential doctrine}
A doctrine $\doctrine{\mC}{P}$ is \textbf{existential} (resp. \textbf{universal}) if, for every $A_1$ and $A_2$ in $\mathcal{C}$ and every projection $\arrow{A_1\times A_2}{{\pr_i}}{A_i}$, $i=1,2$, the functor:
$$ {PA_i}\xrightarrow{{P_{\pr_i}}}{P(A_1\times A_2)}$$
has a left adjoint $\exists_{\pr_i}$ (resp. a right adjoint $\forall_{\pr_i}$), and these satisfy the Beck-Chevalley condition BC: for any pullback diagram:
$$
\quadratocomm{X'}{A'}{X}{A}{{\pr'}}{f'}{f}{{\pr}}
$$
where $\pr$ and $\pr'$ are projections, and for any $\beta$ in $P(X)$ the equality:
$$\exists_{\pr'}P_{f'}\beta= P_f \exists_{\pr}\beta \,\,\,\, \textnormal{    (  resp. }\forall_{\pr'}P_{f'}\beta= P_f \forall_{\pr}\beta\textnormal{  )}$$ holds. \end{definition} \noindent Observe that the inequality $\exists_{\pr'}P_{f'}\beta\leq P_f \exists_{\pr}\beta \textnormal{    (  resp. }\forall_{\pr'}P_{f'}\beta\geq P_f \forall_{\pr}\beta\textnormal{  )}$ of BC in Definition \ref{def existential doctrine} always holds.

We conclude the current subsection recalling from \cite{maiettipasqualirosolini,maiettirosolini13a,maiettirosolini13b} that doctrines form a 2-category $\CatDoc$ where:
\begin{itemize}
\item a \textbf{1-cell} is a pair $(F,\mathfrak{b})$:
$$\twomorphism{\mC}{P}{F}{\mathfrak{b}}{\mD}{R}$$
such that $\arrow{\mC}{F}{\mD}$ is a finite product preserving functor between doctrines $P, R$, and: $$\arrow{P}{\mathfrak{b}}{R F^{\op}}$$ is a natural transformation;
\item a \textbf{2-cell} $(F,\mathfrak{b})\xrightarrow{\theta}(G,\mathfrak{c})$ is a natural transformation $\arrow{F}{\theta}{G}$ such that for every $A$ in $\mC$ and every $\alpha$ in $P(A)$, we have: 
$$\mathfrak{b}_A(\alpha)\leq R_{\theta_A}(\mathfrak{c}_A(\alpha)).$$
\end{itemize} 
We denote as $\ED$ the 2-full subcategory of $\CatDoc$ whose elements are existential doctrines, and whose 1-cells are those 1-cells of $\CatDoc$ which preserve the existential structure. Similarly, we  denote by $\UD$ the 2-full subcategory of $\CatDoc$ whose elements are universal doctrines, and whose 1-cells are those 1-cells of $\CatDoc$ which preserve the universal structure.

From a logical perspective, the intuition is that a 1-cell between doctrines is a generalisation of the notion of set-theoretic model, whereas 2-cells represent morphisms of models. 

\subsection{Internal language of doctrines}
Over the years, category theory and categorical logic have evolved a characteristic form of proof by diagram chasing to establish properties expressible in category theoretic terms. However, in complex cases such arguments can be difficult to construct and hard to follow because of the rather limited forms of expression of a purely category-theoretic language. 

Categorical logic enables the use of richer and more familiar forms of expression meant to establish properties of particular kinds of categories. Indeed, one can define a suitable notion of \emph{internal language}, naming the relevant constituents of the category and then applying a categorical semantics to turn assertions of this language (according to a suitable logic) into categorical statements. Such a procedure has become highly developed in the theory of toposes where the internal language of a topos coupled with the semantics of intuitionistic higher order logic in toposes  enables one to reason about the objects and morphisms of a topos as if they were sets and functions. The notion of internal language is not just a useful instrument to simplify the notation, but it is a powerful instrument that allows us to formally prove a categorical equivalence between doctrines and logical theories.

First, we briefly recall that theories in a given (possibly many-sorted) language over a fragment of first-order logic induce doctrines. Let us assume that we are given a fragment $F$ of first-order logic. Whenever $\mathcal{L}$ is a (possibly) many sorted $F$-language and $\textsc{T}$ is an $\mathcal{L}$-theory, we can define a doctrine $\mathcal{L}\textsc{T}$ over the syntactic category $\mathcal{V}$ associated to $\mathcal{L}$, as described in Section 2. Whenever $P$ is a doctrine over some category $\mathcal{C}$, then $P$ can host \textit{models} of $\textsc{T}$ according to a natural generalisation of Tarski's semantics, which is sound and complete (for it admits the syntactic model) and is formally defined as classic Tarski's semantics for (a fragment of) first-order logic. In fact, a $\mathcal{L}$-structure $S$ in $P$ consists of:

\begin{itemize}
    \item an object of $\mathcal{C}$ for every $\mathcal{L}$-sort;
    \item an arrow of $\mathcal{C}$ (between the appropriate $S$-interpretations of the sorts) for any $\mathcal{L}$-function symbol;
    \item for every $\mathcal{L}$-predicate symbol in some context, an element of the $P$-fibre of the $S$-interpretation of that context.
\end{itemize} 
Then (terms and) formulas are inductively interpreted in $S$ as usual (formally as for traditional Tarski's semantics) and we say that $S$ \textbf{models} some $\mathcal{L}$-sequent $\phi \vdash \psi$ in some given context when it is the case that: $$\phi^S \leq \psi^S$$ in the $P$-fibre of the $S$-interpretation of the given context (here $\alpha^S$ denotes the $S$-interpretation of some formula $\alpha$). According to this notion of semantics, it is the case that the $P$-models of $\textsc{T}$, together with the model morphisms between them, are bijectively (equivalently) induced by the $1$-cells $\mathcal{L}\textsc{T}\to P$ of $\CatDoc$ and the $2$-cells between them, respectively. Thus, the identity over $\mathcal{L}\textsc{T}$ induces the syntactic model of $\textsc{T}$.

Conversely, a language can be defined for a given doctrine $\doctrine{\mC}{P}$. 
%
One can associate to each doctrine $\doctrine{\mC}{P}$ a language $\mathcal{L}_P$ having a sort $A$ for each object $A$ of the base category $\mC$,  an $n$-ary function symbol $\arrow{A_1,\dots,A_n}{f}{A}$ for every morphism $\arrow{A_1\times\dots\times A_n}{f}{A}$ of $\mC$ and an $n$-relation symbol $R \colon A_1,\dots,A_n$ for each element of $P(A_1\times\dots\times A_n)$, all of this for each finite list of objects $A_1,\dots,A_n$ of $\mathcal{C}$ and every object $A$ of $\mC$. The language $\mathcal{L}_P$ is called the \textbf{internal language} of the doctrine $P$. Let $\textsc{T}_P$ be the theory whose sequents $\phi \vdash \psi$ in some context $A$ are precisely those ones such that $\phi \leq \psi$ in $P(A)$. The doctrine $\mathcal{L}_P\textsc{T}_P$ is equivalent to $P$ in $\CatDoc$ and the assignment $P \mapsto (\mathcal{L}_P, \textsc{T}_P)$ extends to a pseudo 2-inverse to the 2-functor $(\mathcal{L},\textsc{T})\mapsto \mathcal{L}\textsc{T}$, in such a way that $\CatDoc$ is equivalent to the theories in some language over $F$ together with the models (of one of them into the other one) and the model morphisms (between them).

Whenever $\doctrine{\mC}{P}$ is a doctrine, we know that $\mathcal{L}_P\textsc{T}_P$ is equivalent to $P$ in $\CatDoc$. Therefore, the doctrine $P$, together with the equivalence $\mathcal{L}_P\textsc{T}_P\to P$, constitutes the syntactic model of its own theory $\textsc{T}_P$ in its own internal language $\mathcal{L}_P$. This fact means that, whenever $\phi$ and $\psi$ are elements of $P(A)$, for some object $A$ of $\mathcal{C}$, it is the case that $\phi \leq \psi$ precisely when $\phi \vdash \psi$ is a sequent of $\textsc{T}_P$ in context $A$. This is precisely why we can deduce properties of $P$ through a purely syntactical procedure: every $\mathcal{L}_P$-sequent corresponds to a categorical statement or a condition involving $P$, and  this is true precisely when that sequent belongs to $\textsc{T}_P$. 

Taking advantage of these equivalent ways of reasoning about doctrines and logic, we define the following notation for this logical syntax, 
which we use extensively in this paper.  We write in the internal language of a doctrine:
\[ a_1: A_1,\dots,a_n:A_n \; | \; \phi(a_1,\dots,a_n)\vdash \psi(a_1,\dots,a_n) \]
instead of:
\[\phi\leq \psi\]
in the fibre $P(A_1\times \cdots\times A_n)$. Similarly, we write:
\[a:A\; |\; \phi(a)\vdash \exists b : B .\psi (a,b) \text{ and } a:A\; |\; \phi(a)\vdash \forall b : B . \psi (a,b) \]
in place of:
\[\phi\leq \exists_{\pr_A} \psi \text{ and } \phi\leq \forall_{\pr_A} \psi \]
in the fibre $P(A)$, where $\pi_A$ is the projection $A\times B \to A$. In fact, if a doctrine $\doctrine{\mC}{P}$ is existential and $\alpha   \in P(A\times B)$ is a formula-in-context: $$a: A,b: B \; |\; \alpha (a,b)$$ then $\exists_{\pi_A}\alpha \in PA$ represents the formula $a: A\;|\; \exists b: B . \alpha(a,b)$ in context $A$. Analogously, if the doctrine $P$ is universal, then $\forall_{\pi_A}\alpha \in PA$ represents the formula $a: A \; |\; \forall b: B .\alpha(a,b)$ in context $A$. 
This interpretation is sound and complete for the usual reasons: this is how usual Tarski semantics can be characterised in terms of categorical properties of the powerset functor $\doctrine{\set}{\mathcal{P}}$.

Also, we write $a:A \; |\; \phi \dashv\vdash \psi$ to abbreviate $a:A \; |\; \phi \vdash \psi$ and $a:A \; |\; \psi \vdash \phi$ and when the type of a quantified variable is clear from the context, we will omit that type for sake of readability. Finally, substitutions via given terms (i.e. reindexings and weakenings) are modelled by pulling back along those given terms. Applications of propositional connectives are interpreted by using the corresponding operations in the fibres of the given doctrine.

\section{Quantifier-free elements and G\"odel doctrines}
One of the fundamental notions of logic and proof theory is the notion of \emph{quantifier-free formula}, and there are countless results built on the possibility of detecting quantifier-free formulae in the literature.  For example, in the Dialectica interpretation, this notion is present at every stage, and we could say the entire translation depends on the fact that, syntactically, we can identify and distinguish formulae with no occurrences of quantifiers.

However, while from a syntactic perspective it is effortless and natural to speak of quantifier-free formulae, abstracting this notion algebraically is not so obvious.
The main problem is that the property of being quantifier-free is totally syntactic, not involving any other entity different from the formula itself we are considering. It does not depend, for example, on the fact that we are working in classical, constructive or intuitionistic logic. It only depends on how a formula is written in a given formal language.

Therefore, if we want to provide a complete categorical presentation of the Dialectica interpretation, capable of representing all its logical details, we have to deal with the problem of representing quantifier-free formulae and find a suitable \emph{universal property} to represent predicates that are quantifier-free categorically.
Notice that quantifier-free elements may satisfy different properties depending on the logical system we are considering, hence if we want to represent these elements via universal properties, we have to relativise this notion to a give system, that is the Dialectica interpretation in our case.

\subsection{Existential and universal free elements}
We discuss a notion to identify those predicates of an existential doctrine: $$\fibration{\mathcal{C}^{\op}}{P}{\pos}$$ which are \emph{free from left-adjoints} $\exists_{\pr}$, and then dualise this notion to define those predicates that are \emph{free from right-adjoints} $\forall_{\pr}$. This idea was originally introduced by Trotta and Maietti \cite{trottamaietti2020} and, independently, by Frey in  \cite{Frey2020}. It was further developed and generalised to the fibrational setting in \cite{trotta_et_al:LIPIcs.MFCS.2021.87}.

\begin{definition}\label{definition (P,lambda)-atomic0}
Let $\fibration{\mathcal{C}^{\op}}{P}{\pos}$ be an existential doctrine and let $A$ be an object of $\mathcal{C}$. A predicate $\alpha$ of the fibre $P(A)$ is said to be an  \textbf{existential splitting} if it satisfies the following \emph{weak} universal property: for every predicate $\beta\in P(A\times B)$ such that $\alpha(a) \vdash \exists b \colon B.\beta(a,b)$ (i.e. $\alpha\leq \exists_{\pr_A}(\beta)$ in category-theoretic notation, where $\arrow{A\times B}{\pr_A}{A}$ is a product projection of $\mathcal{C}$), there exists an arrow $\arrow{A}{g}{B}$ such that: $$\alpha(a)\vdash \beta(a,g(a)) \textnormal{  \, \, ( i.e. } \alpha \leq P_{\angbr{1_A}{g}}(\beta) \textnormal{ in category-theoretic notation ).}$$
\end{definition} Existential splittings stable under re-indexing are called \emph{existential-free elements}. Thus we introduce the following definition:
\begin{definition}\label{definition (P,lambda)-atomic}
Let $\fibration{\mathcal{C}^{\op}}{P}{\pos}$ be an existential doctrine and let $I$ be an object of $\mathcal{C}$. A predicate $\alpha$ of the fibre $P(I)$ is said to be  \textbf{existential-free} if $P_f(\alpha)$ is an existential splitting for every morphism $\arrow{A}{f}{I}$.
\end{definition}

Employing the presentation of doctrines via internal language,  we say that $i:I\; |\; \alpha(i)$ is free from the existential quantifier if, whenever $a:A \;|\; \alpha(f(a))\vdash \exists b: B.\beta (a,b)$ for some term $a:A\;|\; f(a): I$, then there is a term $a:A\; |\; g(a):B$ such that $a:A\; |\; \alpha(f(a))\vdash \beta(a,g(a))$.

Observe that we always have that $a:A\; |\; \beta(a,g(a))\vdash \exists b: B.\beta (a,b)$, in other words $  P_{\langle 1_A,g\rangle}\beta\leq \exists_{\pi_A}\beta$. In fact, it is the case that $\beta \leq P_{\pi_A}\exists_{\pi_A}\beta$ (as this arrow of $P(A \times B)$ is nothing but the unit of the adjunction $\exists_{\pi_A}\dashv P_{\pi_A}$), hence a re-indexing by the term $\langle 1_A,g\rangle$ yields the desired inequality. Therefore, the property that we  require for $i:I\; |\; \alpha(i)$ turns out to be the following: whenever there are proofs of $\exists b: B.\beta(a,b)$ from $\alpha(f(a))$, at least one of them factors through the canonical proof of $\exists b: B.\beta (a,b)$ from $\beta(a,g(a))$ for some term $a:A\; |\; g(a) : B$. 

Requiring the stability under substitution as in Definition \ref{definition (P,lambda)-atomic} is motivated by the fact that, in logic, if a formula is existential-free, and we apply a substitution to this formula, then we obtain again an existential-free formula.
\begin{definition}\label{definition sud-doctrine of existential free elements}
Let $\doctrine{\mC}{P}$ be an existential doctrine. Then we indicate by $\doctrine{\mC}{\exfreedoctrine{P}}$ the subdoctrine of $P$ whose elements of the fibres $\exfreedoctrine{P}(A)$ are existential-free elements of $P(A)$.
\end{definition}

Dualising the previous Definitions \ref{definition (P,lambda)-atomic0} and \ref{definition (P,lambda)-atomic} we get the corresponding ones for the universal quantifier.

\begin{definition}\label{definition (P,lambda)-coatomic0}
Let $\fibration{\mathcal{C}^{\op}}{P}{\pos}$ be a universal doctrine and let $A$ be an object of $\mathcal{C}$. A predicate $\alpha$ of the fibre $P(A)$ is said to be a  \textbf{universal splitting} if it satisfies the following weak universal property: for every predicate $\beta\in P(A\times B)$ such that $\forall b \colon  B.\beta(a,b) \vdash \alpha(a)$, there exists an arrow $\arrow{A}{g}{B}$ such that:
\[    \beta(a,g(a))\vdash \alpha(a). \]
\end{definition}

\begin{definition}\label{definition (P,lambda)-coatomic}
Let $\fibration{\mathcal{C}^{\op}}{P}{\pos}$ be a universal doctrine and let $I$ be an object of $\mathcal{C}$. A predicate $\alpha$ of the fibre $P(I)$ is said to be \textbf{universal-free} if $P_f(\alpha)$ is a universal splitting for every morphism $\arrow{A}{f}{I}$.
\end{definition}

Again, employing the presentation of a doctrine via its internal language, the property we require of the formula $i:I\; |\; \alpha(i)$, so that it is free from universal quantifiers, is that, whenever $a:A\; | \; \forall b: B .\beta (a,b)\vdash \alpha(f(a))$ for some term $a:A \; |\; f(a):I$, then there is a term $a:A\; |\; g(a):B$  such that $a:A\;| \; \beta(a,g(a))\vdash \alpha(f(a))$.


\begin{definition}
Let $\fibration{\mathcal{C}^{\op}}{P}{\pos}$ be an existential doctrine. We say that $P$ has \textbf{enough existential-free predicates} if, for every object $I$ of $\mathcal{C}$ and every predicate $\alpha\in P(I)$, there exist an object $A$ and an existential-free object $\beta$ in $P(I \times A)$ such that $\alpha(i) \dashv \vdash \exists i \colon I.\beta(i,a)$ (i.e.
$\alpha=\exists_{\pr_I}\beta$).
\end{definition}
Analogously, we have the following definition for universal doctrines.
\begin{definition}
Let $\fibration{\mathcal{C}^{\op}}{P}{\pos}$ be a universal doctrine. We say that $P$ has \textbf{enough universal-free predicates} if, for every object $I$ of $\mathcal{C}$ and every predicate $\alpha\in PI$, there exist an object $A$ and a universal-free object $\beta$ in $P(I \times A)$ such that
$\alpha(i) \dashv \vdash \forall i \colon I.\beta(i,a)$.
\end{definition}
\begin{definition}\label{definition sud-doctrine of universal free elements}
Let $\doctrine{\mC}{P}$ be an universal doctrine. Then we indicate by $\doctrine{\mC}{\unfreedoctrine{P}}$ the subdoctrine of $P$ whose elements of the fibres $\unfreedoctrine{P}(A)$ are universal-free element of $P(A)$.
\end{definition}
\subsection{Skolem and G\"odel doctrines}
Building over  the notions corresponding to  quantifier-free elements in doctrines we introduced in the previous section, we now define two particular kinds of doctrines, called \emph{Skolem doctrines} and \emph{G\"odel doctrines}. 

The \emph{Skolem doctrine} are so called because these doctrines satisfy a version of the traditional principle of \emph{Skolemisation}, namely $\forall u \exists x \alpha(u,x)\rightarrow \exists f \forall u \alpha (u, fu)$.
The name \emph{G\"odel doctrine} is chosen because we will prove that these doctrines encapsulates in a pure form the basic mathematical features of the Dialectica interpretation.

\begin{definition}\label{definition skolem doctrine}
 A doctrine $\doctrine{\mC}{P}$ is called a \textbf{Skolem doctrine} if:
 \begin{itemize}
     \item[(i)] the category $\mC$ is cartesian closed;
     \item[(ii)] the doctrine $P$ is existential and universal;
     \item[(iii)] the doctrine $P$ has enough existential-free predicates;
     \item[(iv)] the existential-free objects of $P$ are stable under universal quantification, i.e. if $\alpha\in P(A)$ is existential-free, then $\forall_{\pr}(\alpha)$ is existential-free for every projection $\pr$ from $A$.
 \end{itemize}
 \end{definition}
\begin{remark}
 The last point $(iv)$ of Definition \ref{definition skolem doctrine} implies that, given a Skolem doctrine $\doctrine{\mC}{P}$, the sub-doctrine $\doctrine{\mC}{\exfreedoctrine{P}}$ of existential-free predicates of $P$ as defined in \ref{definition sud-doctrine of existential free elements} is a universal doctrine. From a purely logical perspective, requiring existential-free elements to be stable under universal quantification is quite natural since this can be also read as \emph{if} $\alpha (x)$ \emph{is an existential-free formula, then $\forall x. \alpha (x)$ is again an existential-free formula}.
\end{remark}

\begin{proposition}[Skolemisation]\label{proposition skolemisation}
Every Skolem doctrine $\doctrine{\mC}{P}$ validates the Skolemisation principle:
$$ a: A \; | \; \forall b : B. \exists c : C. \alpha(a,b,c)\dashv\vdash \exists f: C^B. \forall b : B. \alpha (a,b,\textnormal{ev}(f,b))$$ 
where $\alpha$ is any predicate in $P(A\times B\times C)$.
\end{proposition}
\begin{proof}

Let us assume that $a : A\;|\;\gamma(a) \vdash \forall b.\exists c.\alpha(a,b,c)$ for some predicate $\gamma \in P(A)$. By point (iv) of Definition \ref{definition skolem doctrine}, we assume without loss of generality that $\gamma(a)$ is existential-free: otherwise there is an existential-free predicate  $\gamma'$ covering $\gamma(a)$ and we get back to our hypothesis by using that $P$ is existential.

Since $P$ is universal, it is the case that $a : A, b : B\;|\;\gamma(a) \vdash \exists c.\alpha(a,b,c)$ and, being $\gamma(a)$ existential-free: $$a : A, b : B\;|\;\gamma(a) \vdash \alpha(a,b,g(a,b))$$ for some term in context $a : A, b : B \;|\; g(a,b) : C$. 
Being $\mathcal{C}$ cartesian closed, there is a context $f : C^{B}$ together with a term in context $f : C^{B}, b : B \; |\; \textnormal{ev}(f,b): C$ 
such that there is a unique term in context $a: A \;|\; h(a): C^{B}$ satisfying $a : A,b : B \;|\; g(a,b)=\textnormal{ev}(h(a),b): C$. Hence: $$a : A, b : B\;|\;\gamma(a) \vdash \alpha(a,b,\textnormal{ev}(h(a),b))$$ and  $P$ being universal, it is the case that: $$a : A\;|\;\gamma(a) \vdash \forall b.\alpha(a,b,\textnormal{ev}(h(a),b)).$$ 
Finally, since: $$a : A\;|\;\forall b.\alpha(a,b,\textnormal{ev}(h(a),b))\vdash \exists f. \forall b.\alpha(a,b,\textnormal{ev}(f,b))$$ (this holds for any predicate $\delta(a,-)$ in place of the predicate $\forall b.\alpha(a,b,\textnormal{ev}(-,b))$) we conclude that: $$a : A\;|\;\gamma(a) \vdash \exists f. \forall b.\alpha(a,b,\textnormal{ev}(f,b)).$$ 
We are done by taking $\forall b.\exists c.\alpha(a,b,c)$ as the predicate $\gamma(a)$.
\end{proof}

\begin{definition}\label{goedel}
A doctrine $\doctrine{\mC}{P}$ is called a \textbf{G\"odel doctrine} if:
\begin{enumerate}
    \item[(i)] $P$ is a Skolem doctrine;
    \item[(ii)] the sub-doctrine $\doctrine{\mathcal{C}}{\exfreedoctrine{P}}$ of the existential-free predicates of $P$ has enough universal-free predicates.
\end{enumerate}
\end{definition}

Now we have all the tools needed to introduce the notion of \emph{quantifier-free predicate} in the categorical setting of G\"odel doctrines.
\begin{definition}\label{definition quantifier-free element}
An element $\alpha$ of a fibre $P(A)$ of a G\"odel doctrine $P$ that is both an existential-free predicate of $P$ and a universal-free predicate in the sub-doctrine $\exfreedoctrine{P}$ of existential-free elements of $P$ is called a \textbf{quantifier-free predicate} of $P$. The sub-doctrine of quantifier-free elements is denoted by $\doctrine{\mC}{\quantfreedoctrine{P}}$.
\end{definition}
Therefore, given a G\"odel doctrine $\doctrine{\mC}{P}$, we have the following canonical inclusions of doctrines:

\[ \xymatrix@+1pc{
\quantfreedoctrine{P} \ar@{^{(}->}[r]^{\iota_1} &\exfreedoctrine{P} \ar@{^{(}->}[r]^{\iota_2} & P
}\]
where $\injarrow{\quantfreedoctrine{P} }{\iota_1}{\exfreedoctrine{P}}$ is a morphism of doctrines, while $\injarrow{\exfreedoctrine{P} }{\iota_2}{P}$ is a morphism of universal doctrines.
\begin{remark}
Notice that a universal-free element of the sub-doctrine $\exfreedoctrine{P}$ of a given G\"odel doctrine $P$, may not be a universal-free element in the whole doctrine $P$, because the universal property of being universal-free is relative only to the elements of $\exfreedoctrine{P}$. 
Therefore, the quantifier-free elements of $P$ as established in Definition \ref{definition quantifier-free element} \emph{are not the existential and universal free elements} of $P$. 
\end{remark}

In order to simplify the notation, but also to make clear the connection with the logical presentation in the Dialectica interpretation, for a given G\"odel doctrine $\doctrine{\mC}{P}$ we will use the notation $\alpha_D$ to indicate an element $\alpha$ of $\quantfreedoctrine{P}$, i.e. a quantifier-free predicate. 

\begin{theorem}\label{theorem prenex normal form}
Let $\doctrine{\mathcal{C}}{P}$ be a G\"odel doctrine, and let $\alpha$ be an element of $P(A)$. Then there exists a quantifier-free predicate $\alpha_D$ of $P(I\times U\times X)$ such that:
\[i:I\; |\; \alpha (i) \dashv\vdash \exists u: U.\forall x: X. \alpha_D (i,u,x).\]
\end{theorem}
\begin{proof}
By definition of a G\"odel doctrine, in particular since $P$ has enough-existential free objects, there exists an existential-free element $\beta\in P(A\times U)$ such that $i:I\; |\; \alpha (i) \dashv\vdash \exists u: U. \beta (i,u)$. Then, since the subdoctrine of existential-free elements has enough-universal free elements, we can conclude that there exists a quantifier-free predicate $\alpha_D$ of $P(I\times U\times X)$ such that $i:I\; |\; \alpha (i) \dashv\vdash \exists u: U.\forall x: X. \alpha_D (i,u,x)$.
\end{proof}
Theorem \ref{theorem prenex normal form} shows that in a G\"odel doctrine every formula admits a presentation of the precise form used in the Dialectica translation.

The next result establishes the precise connection between G\"odel doctrines and the Dialectica interpretation. 
Employing the properties of a G\"odel doctrine, we can provide a complete categorical description and presentation of the chain of equivalences involved in the Dialectica interpretation of implicational formulae. In particular, we  show that the crucial steps where (IP) and (MP) are applied are represented categorically via the notions of existential-free element and universal-free element.

\begin{theorem}\label{theorem principal 1}
Let $\doctrine{\mathcal{C}}{P}$ be a G\"odel doctrine. Then for every $\psi_D\in P(I\times U\times X )$ and $\phi_D\in P(I\times V\times Y)$ quantifier-free predicates of $P$ we have that:
\[ i:I\; |\; \exists u.\forall x . \psi_D (i,u,x)\vdash \exists v.\forall y.  \phi_D (i,v,y)\]
if and only if there exist $\arrow{I\times U}{f_0}{V}$ and $\arrow{I\times U\times Y}{f_1}{X}$ such that:
\[i:I,u:U,y:Y\; |\; \psi_D(i,u,f_1(i,u,y))\vdash \phi_D (i,f_0(i,u),y).\]

\end{theorem}
\begin{proof}
Let us consider two quantifier-free predicates $\psi_D\in P(I \times U\times X)$ and $\phi_D\in P(I \times V\times Y)$ of the G\"odel doctrine $P$.
The following equivalence follows by definition of left adjoint functor (for sake of readability we omit the types of quantified variables):
\begin{align*}
    i: I \; | \; \exists u.\forall x.  \psi_D (i,u,x)\vdash  \exists v.\forall y.  \phi_D (i,v,y)  &\iff \\
  i:I, u: U \; | \; \forall x.  \psi_D (i,u,x)\vdash  \exists v.\forall y.  \phi_D (i,v,y) 
\end{align*} 
Now we employ the fact that the predicate $ \forall x.  \psi_D (i,u,x)$ is existential-free in the G\"odel doctrine, obtaining that there exists an arrow $\arrow{I\times U }{f_0}{V}$, such that:
\begin{align*}
     i:I,u:U \; | \; \forall x.  \psi_D (i,u,x)\vdash  \exists v.\forall y . \phi_D (i,v,y) & \iff\\
      i:I,u:U \; | \;   \forall x  .\psi_D (i,u,x)\vdash   \forall y . \phi_D (i,f_0(i,u),y)
       \end{align*}
Then, since the universal quantifier is right adjoint to the weakening functor, we have that:
\begin{align*}
       i:I,u:U \; | \;  \forall x . \psi_D (i,u,x)\vdash   \forall y . \phi_D (i,f_0(i,u),y) & \iff\\
      i:I,u:U , y:Y\; | \;   \forall x . \psi_D (i,u,x)\vdash    \phi_D (i,f_0(i,u),y).
       \end{align*}
Now we employ the fact that  $ \phi_D (i,f_0(u),y)$ is universal-free in the subdoctrine of existential-free elements of $P$. Notice that since $\psi_D (i,u,x)$ is a quantifier-free element of the G\"odel doctrine, we have that $\forall x.  \psi_D (i,u,x)$ is existential free. 
Recall that this follows from the fact that in every G\"odel doctrine, existential-free elements are stable under universal quantification (this is the last point of Definition \ref{goedel}). Therefore, we can conclude that there exists an arrow $\arrow{I\times U\times Y}{f_1}{X}$ of $\mathcal{C}$ such that:
\begin{align*}
           i:I, u:U , y:Y\; | \;  \forall x.  \psi_D (i,u,x)\vdash    \phi_D (i,f_0(i,u),y)  \iff\\
            i:I, u:U , y:Y\; | \;  \psi_D (i,u,f_1(i,u,y))\vdash    \phi_D (i,f_0(i,u),y)
\end{align*}
Then, combining the first and the last equivalences, we obtain the following equivalence:
\begin{center}
   $i:I\;|\;  \exists u.\forall x.  \psi_D (i,u,x)\vdash  \exists v.\forall y.  \phi_D (i,v,y)  \iff$
    \text{there exist $(f_0,f_1)$ s.t. }
  $i:I, u:U , y:Y\; | \;  \psi_D (i,u,f_1(i,u,y))\vdash    \phi_D (i,f_0(i,u),y).$
\end{center} 

\end{proof}
Notice that in Theorem \ref{theorem principal 1}, the arrow $\arrow{I\times U }{f_0}{V}$ represents the \emph{witness function}, i.e. it assigns to every witness $u$ of the hypothesis a witness $f_0(i,u)$ of the thesis, while the arrow $\arrow{I\times U\times Y}{f_1}{X}$ represents the \emph{counterexample function}. While the witness function $f_0(i,u)$ depends on the witness $u$ of the hypothesis, the counterexample function $f_1(i,u,y)$ depends on a witness of the hypothesis and on  a counterexample of the thesis. This is a  natural fact because, under the constructive point of view, a counterexample has to be relative to a witness validating the thesis.

Therefore, Theorem \ref{theorem principal 1} shows that the notion of G\"odel doctrine encapsulates in a pure form the basic mathematical feature of the Dialectica interpretation, namely its interpretation of implication, which corresponds to the existence of functionals of types $f_0:I\times U\to V$ and $f_1:I\times U\times Y\to X$ as described. One should think of this as saying that a proof of a formula of the form $\exists u.\forall x . \psi_D (i,u,x)\rightarrow \exists v.\forall y.  \phi_D (i,v,y)$ is obtained by transforming it to: $$\forall u.\exists v.\forall y. \exists x. (\psi_D (i,u,x)\rightarrow   \phi_D (i,v,y))$$ 
by means of the Principle of Independence of Premises and Markov Principle, and then Skolemising twice. 

Therefore, combining Theorems \ref{proposition skolemisation}, \ref{theorem prenex normal form} and \ref{theorem principal 1} we obtain  strong evidence that the notion of G\"odel doctrine really provides the categorical abstraction of the main concepts involved in the Dialectica translation. We discus this in more details in the next subsection.

\section{A characterisation of Dialectica doctrines}
The concept of Dialectica category was originally introduced by de Paiva \cite{dePaiva1989dialectica}, and it was generalised to the fibrational setting by Hofstra \cite{hofstra2011}. 

Let us briefly recall the notion of a Dialectica category $\dial{\mathcal{C}}$ associated to a finitely complete category $\mathcal{C}$ (see \cite{dePaiva1989dialectica} for further details):
\begin{itemize}
    \item An \textbf{object} of $\dial{\mathcal{C}}$ is a triple $(U,X,\alpha)$, where $\alpha$ is a subobject of $X \times U$ in $\mathcal{C}$. We think of such a triple as a formula $\exists u.\forall x. \alpha(u,x)$.
    \item An \textbf{arrow} from $\exists u.\forall x. \alpha(u,x)$ to $\exists v.\forall y. \beta(v,y)$, for two objects $(U,X,\alpha)$ and $(V,Y,\beta)$ in $\dial{\mathcal{C}}$ is a pair $(\arrow{U}{F}{V},\,\arrow{U\times Y}{f}{X})$ of arrows of $\mathcal{C}$, i.e. a pair: $$(u \colon U \;|\;F(u) : V,\;\;\;u \colon U, y \colon Y \;|\;f(u,y): X)$$ of terms in context (as usual, we are thinking of $\mathcal{C}$ as the category of contexts associated to some type theory), satisfying the condition: $$\alpha(u,f(u,y))\leq \beta(F(u),y)$$ between the reindexed subobjects, where the squares: $$\quadratocomm{\alpha(u,f(u,y))}{\alpha}{U\times Y}{U \times X}{{}}{}{}{{\langle \pr_U,f\rangle}}\;\;\;\quadratocomm{\beta(F(u),y)}{\beta}{U\times Y}{V \times Y}{{}}{}{}{{F \times 1_Y}}$$ are pullbacks.
\end{itemize}

Our notation is motivated by the notion of \textit{internal language} of a doctrine (see Section 2.2). In fact, a finitely complete category $\mathcal{C}$ is nothing but an instance of a doctrine, if we look at $\mathcal{C}$ itself as a category of contexts associated to some type theory and at the subobjects $\alpha$ of a given object $X$ of $\mathcal{C}$ as the predicates $\alpha(x)$ in context $x \colon X$.

The notions of object and arrow of $\dial{\mathcal{C}}$ are motivated by G\"odel's notion of Dialectica interpretation (see \cite{godel58}), and in particular by its action on formulas in the language of arithmetic of the form $A \to B$. In detail, we recall that one can define a many-sorted language $\mathcal{L}$ together with an inductive notion of a formula $A^D$ of $\mathcal{L}$, whenever $A$ is a given formula in the language of arithmetic. 
The formula $A^D$ is in prenex normal form: $$\exists \Vec{u}.\forall \Vec{x}. A_D(\Vec{u},\Vec{x})$$ for some inductively defined quantifier-free formula $A_D$ of $\mathcal{L}$. Finally, there exists an $\mathcal{L}$-theory $\textsc{T}$, called System $\textsc{T}$, that enjoys the quantifier-elimination property and that allows the interpretation $(-)^D$ to satisfy the following:

\begin{theorem}[G\"odel 1958, Soundness] Let $A$ be a formula in the language of arithmetic. Whenever $\textsc{HA} \vdash A$, where $\textsc{HA}$ is Heyting's Arithmetic, then $\textsc{T}\vdash A^D$ by means of an application of the rules of introduction of quantifiers to $A_D$, that is: $$\textsc{T}\vdash A_D(\Vec{t},\Vec{x})$$ for some (finite sequence of) closed terms $\Vec{t}$.
\end{theorem}

\noindent
in such a way that a result of relative consistency of $\textsc{HA}$ holds: Heyting's Arithmetic is consistent, provided that the System $\textsc{T}$ (which is quantifier-free in $\mathcal{L}$) is.

We conclude the present subsection by recalling what $(A\to B)^D$ looks like, so that the notion of $\dial{\mathcal{C}}$ is justified, and refer the reader to \cite{godel58,Troelstra1973} for further details. If $A$ and $B$ are formulas in the language of arithmetic, then $(A\to B)^D$ is the formula: $$\exists F.\exists f.\forall u.\forall y.(\;A_D(u,f(u,y)) \to B_D(F(u),y)\;)$$ of $\mathcal{L}$. Hence our notion of arrow of $\dial{\mathcal{C}}$ is nothing but a categorical way of expressing the action of $(-)^D$ on $A \to B$.

As we anticipated, the notion of Dialectica category was generalised to an arbitrary fibration/doctrine by Hyland~\cite{Hyland2002}, Biering~\cite{Biering_dialecticainterpretations} and Hofstra~\cite{hofstra2011}, for both a proof-irrelevant and a proof-relevant predicative settings. 
Here we consider the proof-irrelevant version of this construction, associating a doctrine $\dial{P}$ called a \emph{dialectica doctrine} to a given doctrine $P$.

\subsection{Dialectica doctrines}
The notions of Dialectica category and Dialectica fibration are introduced in terms of instances of a free construction called the \emph{Dialectica construction}, i.e. a category is called a \emph{Dialectica category} if it is the output of the Dialectica construction. 

From a modern categorical perspective,  Dialectica categories or Dialectica fibrations are the free algebras of the Dialectica pseudo-monad $\dial{-}$, described by Hofstra \cite{hofstra2011}. Here we will deal with the proof-irrelevant version of such construction:

\medskip
 
 \noindent
\textbf{Dialectica construction.} Let $\doctrine{\mathcal{C}}{P}$ be a doctrine whose base category $\mC$ is cartesian closed. The \textbf{dialectica doctrine} $\doctrine{\mathcal{C}}{\dial{P}}$ is defined as the functor sending:\begin{itemize}
    \item an object $I$ into the poset $\dial{P}(I)$ defined as follows:
\begin{itemize}
\item \textbf{objects} are quadruples $(I,U,X,\alpha)$ where $I,U$ and $X$ are objects of the base category $\mathcal{C}$ and $\alpha\in P(I\times U\times X)$;
\item \textbf{partial order:} we stipulate that $(I, U,X,\alpha)\leq (I,V,Y,\beta)$ if there exists a pair $(f_0,f_1)$, where $\arrow{I\times U}{f_0}{V}$ and  $\arrow{I\times U\times Y}{f_1}{X}$  are morphisms of $\mathcal{C}$  such that: $$\alpha(i,u,f_1(i,u,y))\leq \beta (i,f_0(i,u),y).$$ \end{itemize}
\item an arrow $\arrow{J}{g}{I}$ into the poset morphism $\dial{P}(I)\to \dial{P}(J)$ sending a predicate $(I,U,X,\alpha)$ to the predicate: $$(J,U,X,\alpha(g(j),u,x)).$$
\end{itemize} 
\begin{remark}
Let $\doctrine{\mathcal{C}}{P}$ be a doctrine and let $I$ be any object of $\mathcal{C}$. Then the poset $\dial{P}(I)$ is isomorphic to the poset reflection of the dialectica category associated to some category.
\end{remark}

\subsection{Dialectica doctrines via quantifier completions}

This subsection is devoted to providing a categorical presentation of the notion of Dialectica doctrine. Our aim is to connect the notion of Dialectica construction to the one of a G\"odel doctrine and show that, under certain hypotheses, 
these notions are equivalent.
In order to show this, we ask ourselves when is it the case that a doctrine is an instance of a dialectica construction and, in this case, which doctrine do we need to complete in order to go back to the one we started from.

The main result we need  is the following statement. (Here $\uncomp{Q}$ and $\excomp{Q}$ denote the universal and the existential completions of any doctrine $Q$.  We are going to recap these notions later.)

\begin{proposition}[Hofstra \cite{hofstra2011}]
If $\doctrine{\mathcal{C}}{P}$ is a doctrine, then there is an isomorphism: $$\dial{P}\cong \excomp{(\uncomp{P})}$$ which is natural in $P$.
\end{proposition}

We briefly recall the notion of existential completion of a doctrine, see \cite{trotta2020} for more details:

\medskip

\noindent\textbf{Existential completion}. Let $\doctrine{\mC}{P}$ be a pos-doctrine. The \textbf{existential completion} $\doctrine{\mC}{\excomp{P}}$ of $P$ is the doctrine such that, for every object $A$ of $\mC$, the poset $\excomp{P}(A)$ is defined as follows:
\begin{itemize}
\item \textbf{objects:} triples $(A,B,\alpha)$, where $A$ and $B$ are objects of $\mC$ and $\alpha$ is a predicate in $P(A\x B)$.
\item \textbf{order:}  $(A,B,\alpha)\leq (A,C,\beta)$ if there exists an arrow $\arrow{A\x B}{f}{C}$ of $\mC$ such that: $$ \alpha(a,b)\vdash \beta(a,f(a,b)) \textnormal{  \, \, ( i.e. } \alpha \leq P_{\angbr{\pr_A}{f}}(\beta) \textnormal{ )}$$
in $P(A\times B)$ (here $\arrow{A\x B}{\pr_A}{A}$ is the projection on $A$).
\end{itemize}
Whenever $f$ is an arrow $A \to C$ of $\mathcal{C}$, the functor $\arrow{\excomp{P}(C)}{\excomp{P}_f}{\excomp{P}(A)}$ sends an object $(C,D,\gamma)$ of $\excomp{P}(C)$ to the object: \begin{center}
    $(A,D,\gamma(f(a),d))$ \, \, \, ( \, i.e. $(A,D,P_{\langle f \pr_A,\pr_D \rangle}(\gamma))$ \, )
\end{center} of $\excomp{P}(A)$ (here $\pr_A, \pr_D$ are the projections from $A\times D$).
\medskip

We think of a triple $(A,B,\alpha)$ in $\excomp{P}(A)$ as the predicate $(\exists b \colon B)\alpha(a,b)$. This construction provides a free completion, i.e. 
it extends to a 2-functor which is left adjoint to the corresponding forgetful functor. We remind that the associated monad happens to be lax-idempotent. Analogously, let us remind the notion of universal completion of a doctrine:

\medskip

\noindent\textbf{Universal completion}. Let $\doctrine{\mC}{P}$ be a pos-doctrine. The \textbf{universal completion} $\doctrine{\mC}{\uncomp{P}}$ of $P$ is the doctrine such that, for every object $A$ of $\mC$, the poset $\uncomp{P}(A)$ is defined as follows:
\begin{itemize}
\item \textbf{objects:} triples $(A,B,\alpha)$, where $A$ and $B$ are objects of $\mC$ and $\alpha$ is a predicate in $P(A\x B)$.
\item \textbf{order:}  $(A,B,\alpha)\leq (A,C,\beta)$ if there exists an arrow $\arrow{A\x C}{g}{B}$ of $\mC$ such that: $$ \alpha(a,g(a,c))\vdash \beta(a,c)$$ in $P(A\times C)$.
\end{itemize}
Again, if $f$ is an arrow $A \to C$ of $\mathcal{C}$, the functor $\arrow{\uncomp{P}(C)}{\uncomp{P}_f}{\uncomp{P}(A)}$ sends an object $(C,D,\gamma)$ of $\uncomp{P}(C)$ to the object $(A,D,\gamma(f(a),d))$ of $\uncomp{P}(A)$.
\medskip

We think of a triple $(A,B,\alpha)$ in $\uncomp{P}(A)$ as the predicate $(\forall b \colon B)\alpha(a,b)$. Finally, this construction provides a free completion, i.e. it extends to a 2-functor which is right adjoint to the obvious forgetful functor inducing a colax-idempotent monad. The universal and the existential completions of a given doctrine $P$ are related by the following natural isomorphism: \begin{equation}\label{f}
    \uncomp{P} \cong (-)^{\op} \excomp{((-)^{\op} P)}
\end{equation} where $(-)^{\op}$ is the functor $\pos \to \pos$ which inverts the order of any poset (see \cite{trottapasquale2020}).

We recall that the existential and universal completions are really well-behaved, in the sense that the instances of such a completion can be internally characterised without referring to a further doctrine that the doctrine we are given is a completion of. That is what we talk about in the following statements, contained in Proposition \ref{ex-un-char} (we refer to \cite{trotta2020,trottamaietti2020} for a proof of this statement and to \cite{trotta_et_al:LIPIcs.MFCS.2021.87} for its extension to its proof-relevant version), where the latter follows from the former by means of the natural isomorphism (\ref{f}):

\begin{proposition}\label{ex-un-char}
Let $\doctrine{\mathcal{C}}{P}$ be a doctrine.
Assume that $P$ is existential. Then $P$ is an existential completion of some other doctrine $P'$ precisely when it has \textnormal{enough existential-free predicates}, i.e. when, for every predicate $a \colon A \; |\; \alpha(a)$ of $P$, there is an existential-free predicate: $$a\colon A,b\colon B\; |\;\beta(a,b)$$ of $P$ such that $\alpha(a) \cong (\exists b \colon B)\beta(a,b)$ in $P(A)$. In this case, such a doctrine $P'$ is the full sub-doctrine $\exfreedoctrine{P}$ of $P$ whose predicates are the existential-free ones of $P$.

Analogously, if $P$ is a   universal doctrine, then $P$ is a universal completion of some doctrine $P'$ precisely when it has \textnormal{enough universal-free predicates}, i.e. when, for every predicate $a \colon A \; |\; \alpha(a)$ of $P$, there is a universal-free predicate: $$a\colon A,b\colon B\; |\;\beta(a,b)$$ of $P$ such that $\alpha(a) \cong (\forall b \colon B)\beta(a,b)$ in $P(A)$. In this case, such a doctrine $P'$ is the full sub-doctrine $\unfreedoctrine{P}$ of $P$ whose predicates are the universal-free predicates of $P$.
\end{proposition}

By means of the Proposition \ref{ex-un-char} above, the following result follows. This provides a characterisation of the free-algebras of the monad $\dial{-}$.

\begin{theorem}\label{char}
Let us assume that the category $\mathcal{C}$ is cartesian closed. Then the doctrines $\doctrine{\mathcal{C}}{P}$ that are dialectica completions of some doctrine $P'$ are precisely those that are G\"odel doctrines. Moreover, in this case, such a doctrine $P'$ is the full sub-doctrine $\quantfreedoctrine{P}$ of the quantifier-free predicates of $P$.
\end{theorem}

We end the current session with the following:

\begin{remark}
The existential completion of a universal doctrine whose base is cartesian closed happens to be universal as well (see \cite{hofstra2011} for more details). Therefore, by Theorem \ref{char}, it is the case that G\"odel doctrines happen to be both existential and universal. 

Complementing this description, one can also look at the dialectica completion as a procedure to add the existential and the universal quantifications to the predicative part of a type theory containing at least the simply typed lambda calculus. Trotta and Spadetto~\cite{trottapasquale2020}  analyse which logical structure, that we might assume to be already present in the predicative part of our type theory, is preserved --or at least maintained-- by this procedure.
\end{remark}

\section{Logical principles in G\"odel doctrines}

G\"odel doctrines provide a categorical framework that generalises the principal concepts underlying the Dialectica translation, such as the existence of witness and counterexample functions, whenever we have an implication $i:I\; |\; \exists u.\forall x . \psi_D (u,x,i)\vdash \exists v.\forall y . \phi_D (v,y,i)$. 
The key idea is that, intuitively, the notion of \emph{existential-quantifier-free} objects can be seen as a reformulation of the \emph{principle of independence of premises}, while \emph{product-quantifier-free} objects can be seen as a reformulation of \emph{Markov principle}. Notice that in the proof of Theorem \ref{theorem principal 1} existential and universal free elements play the same role that (IP) and (MP) have in the Dialectica interpretation of implicational formulae.

The main goal of this section is to formalise this intuition, showing the exact connection between the principles (IP) and (MP) and Gödel first-order hyperdoctrines.

\medskip

\noindent
\textbf{Notation.} to denote the rule-version of the logical principles we consider, we will add R- to the name of principles. 
For example we will denote by (R-IP) the rule:
\[ \top\vdash \theta\rightarrow \exists u.\eta(u)\ \text{ implies } \top\vdash \exists u.( \theta\rightarrow \eta(u))\]
corresponding to the principle of independence of premise (IP):
\[ \top \vdash (\theta\rightarrow \exists u.\eta(u))\rightarrow \exists u.( \theta\rightarrow \eta(u))\]
and similarly we will use (R-MP) for Markov rule.
\subsection{Dialectica interpretation of implication}\label{section introducing dialectica, MP and IP}

Gödel's Dialectica interpretation \cite{godel58,goedel1986} 
associates to each formula $\phi$ in the language of arithmetic its  \emph{Dialectica interpretation} $\phi^D$, i.e. a formula of the form: 
\[\phi^D =\exists u.\forall x .\phi_D\]
where $\phi_D$ is a quantifier-free formula in the language of system T, trying to be \emph{as constructive as possible}. 
The associations $(-)^D$ and $(-)_D$ are defined inductively on the structure of the formulae, and we refer to \cite{godel58,goedel1986} for a complete description.
The most complicated clause of the translation (and, in G\"odel’s words, “the most important one”) is the definition of the translation of the implication connective $(\psi\rightarrow \phi)^D$. This involves two logical principles which are usually not acceptable from an intuitionistic point of view, namely a form of the \emph{Principle of Independence of Premise} (IP) and a generalisation of \emph{Markov Principle} (MP).
The interpretation is given by:
\[(\psi\rightarrow \phi)^D=\exists V,X.\forall u,y. (\psi_D(u,X(u,y))\rightarrow \phi_D (V(u),y)).\]
The motivation provided 
in the collected works of G\"odel 
for this translation
is that given a witness $u$ for the hypothesis $\psi_D$ one should be able to obtain a witness for the conclusion $\phi_D$, i.e. there exists a function $V$ assigning a witness $V(u)$ of $\phi_D$ to every witness $u$ of $\psi_D$. Moreover, this assignment has to be such that from a counterexample $y$ of the conclusion $\phi_D$ we should be able to find a counterexample $X(u,y)$ to the hypothesis $\psi_D$.  This transformation of counterexamples of the conclusion into counterexamples for the hypothesis is what gives Dialectica its essential bidirectional character.

We first recall the technical details behind the translation of $(\psi\rightarrow \phi)^D$ (\cite{goedel1986}) showing the precise points in which we have to employ the non-intuitionistic principles (MP) and (IP).
First notice that $\psi^D\rightarrow \phi^D$, that is:
\begin{equation}\label{eq 1 implication}
    \exists u.\forall x.\psi_D(u,x)\rightarrow \exists v.\forall y. \phi_D(v,y)
\end{equation}
is  equivalent to:
\begin{equation}\label{eq 2 implication}
    \forall u.( \forall x.\psi_D(u,x)\rightarrow \exists v.\forall y. \phi_D(v,y)).
\end{equation}
If we apply a special case of the \textbf{Principle of Independence of Premise}, namely:
\begin{align*}
\tag{IP*}
( \forall x.\theta(x)\rightarrow \exists v.\forall y .\eta(v,y))\rightarrow \exists v.( \forall x.\theta(x)\rightarrow \forall y. \eta(v,y))
\end{align*}
we obtain that \eqref{eq 2 implication} is equivalent to:
\begin{equation}\label{eq 3 implication}
    \forall u.\exists v.( \forall x.\psi_D(u,x)\rightarrow \forall y. \phi_D(v,y)).
\end{equation}
Moreover, we can see that this is equivalent to:
\begin{equation}\label{eq 4 implication}
    \forall u.\exists v.\forall y. ( \forall x.\psi_D(u,x)\rightarrow  \phi_D(v,y)).
\end{equation}
The next equivalence is motivated by a generalisation of \textbf{Markov’s Principle}, namely:
\begin{equation*}
\tag{MP}
    \neg \forall x. \theta(u,x) \rightarrow \exists  x. \neg \theta(u,x).
\end{equation*}
By applying (MP) we obtain that \eqref{eq 4 implication} is equivalent to:
\begin{equation}\label{eq 5 implication}
    \forall u.\exists v.\forall y.\exists x. ( \psi_D(u,x)\rightarrow  \phi_D(v,y)).
\end{equation}
To conclude that $\psi^D\rightarrow \phi^D=(\psi\rightarrow \phi)^D$ we have to apply the \textbf{Axiom of Choice} (or \textbf{Skolemisation}), i.e.:
 \begin{equation*}
 \tag{AC}
     \forall y. \exists x. \theta(y,x) \rightarrow \exists V. \forall y. \theta(y,V(y))
 \end{equation*}
 twice, obtaining that \eqref{eq 5 implication} is equivalent to:
 \begin{equation*}
   \exists V,X.\forall u,y. (\psi_D(u,X(u,y))\rightarrow \phi_D (V(u),y)).
 \end{equation*}
This analysis (from G\"odel's Collected Works, page 231) highlights the key role the principles (IP), (MP) and (AC)
play in the Dialectica interpretation of implicational formulae. 
The role of the axiom of choice (AC)
has been discussed from  a categorical perspective both by Hofstra \cite{hofstra2011} and in our previous work \cite{trotta_et_al:LIPIcs.MFCS.2021.87}.
We re-examine the two principles (IP) and (MP) in the next subsections, following what we discussed in \cite{trotta_et_al:LIPIcs.MFCS.2021.87}.

\subsection{Independence of Premise}\label{subsec: independence of premise}
In logic and proof theory, the \textbf{Principle of Independence of Premise} states that:
\begin{equation*}
       (\theta\rightarrow \exists u.\eta(u))\rightarrow \exists u.( \theta\rightarrow \eta(u))
\end{equation*}
where $u$ is not a free variable of $\theta$. While this principle is valid in classical logic (it follows from the law of the excluded middle), it does not hold in intuitionistic logic, and it is not generally accepted constructively \cite{AvigadFeferma70}.
The reason why the principle (IP) is not generally accepted constructively is that, from a constructive perspective,  turning any proof of the premise $\phi$ into a proof of $\exists u .\eta(u)$ means turning a proof of $\theta$ into a proof of $\eta (t) $ where $t$ is a witness for the existential quantifier depending on the proof of $\theta$.  In particular, the choice of the witness \emph{depends} on the proof of the premise $\theta$, while the (IP) principle tell us, constructively, that the witness can be chosen independently of any proof of the premise $\theta$.

In the Dialectica translation we only need a particular version of the (IP) principle:
\begin{align*}
\tag{IP*}
( \forall y.\theta(y)\rightarrow \exists u.\forall v. \eta(u,v))\rightarrow \exists u.( \forall y.\theta(y)\rightarrow \forall v. \eta(u,v))
\end{align*}
which means that we are asking (IP) to hold not for every formula, but only for those formulas of the form $\forall y .\theta(y)$ with $\theta$ quantifier-free. 
We  recall a useful generalisation of the (IP*) principle, namely: 
\begin{align*}
\tag{IP}
( \theta \rightarrow \exists u . \eta(u))\rightarrow \exists u.( \theta \rightarrow   \eta(u))
\end{align*}
where $\theta$ is $\exists$-free, i.e. $\theta$  contains neither existential quantifiers nor disjunctions (of course, it is also assumed that $u$ is not a free variable of $\theta$). Therefore, the condition that (IP) holds for every formula of the form $\forall y .\theta (y)$ with $\theta (y)$ quantifier-free is replaced by asking that it holds for every formula \emph{free from the existential quantifier}.


\subsection{Markov Principle}
\textbf{Markov Principle} is a statement that originated in the Russian school of constructive mathematics.
Formally, Markov principle is usually presented as the statement:
\begin{equation*}
  \neg \neg \exists x.\phi (x) \rightarrow \exists x. \phi (x )
\end{equation*}
where $\phi$ is a quantifier-free formula. Thus, MP in the Dialectica interpretation, namely:
\begin{equation*}
    \tag{MP}
    \neg \forall x.\phi (x)\rightarrow \exists x. \neg \phi (x)
\end{equation*}
with $\phi(x)$ a quantifier-free formula, can be thought of as a generalisation of the Markov Principle above. As  remarked in \cite{AvigadFeferma70}, the reason why  MP is not generally accepted in constructive mathematics is that in general there is no reasonable way to choose constructively a witness $x$ for $\neg \phi (x)$ from a proof that $\forall x. \phi (x)$ leads to a contradiction.
However, in the context of Heyting Arithmetic, i.e. when $x$ ranges over the natural numbers, one can prove that these two formulations of Markov Principle are equivalent. More details about the computational interpretation of Markov Principle can be found in \cite{manighetti2016computational}. 

A natural generalisation of (MP) is given by the following principle, that we call \textbf{Modified Markov principle}: i.e. whenever $\beta_D \in P(A)$ is a quantifier-free predicate and $\alpha \in P(A\times B)$ is an existential-free predicate, it is the case that: 
\begin{equation*}
    \tag{MMP}
    (\forall x. \phi(x) \rightarrow \psi(y)) \rightarrow \exists x.(\phi(x)\rightarrow \psi (y))
    \end{equation*}
where $\psi(y)$ is quantifier free, $\alpha (x)$ is existential-free and the variable $x$ does not occur free in $\psi(y)$. Notice that (MP) is obtained by (MMP) by replacying $\psi(y)$ with $\bot$.

\subsection{Gödel hyperdoctrines}
The main purpose of this subsection is to show in which sense the principles (IP) and (MP) are satisfied in a Gödel hyperdoctrine. In  this subsection we deal with their rule versions. First we have to equip Gödel doctrines with the appropriate Heyting structure in the fibres in order to be able to formally express these principles. Therefore, we have to consider Gödel hyperdoctrines.

\begin{definition}\label{definition godel hyperdoctrine}
A hyperdoctrine $\hyperdoctrine{\mathcal{C}}{P}$ is said a \textbf{Gödel hyperdoctrine} when $P$ is a Gödel doctrine.

\end{definition}

Notice that from a logical perspective, one might want the quantifier-free predicates to be closed with respect to all the propositional connectives, since this is what happens in logic. However, for sake of generality, we start requiring just the Heyting structures of the fibres and studying the logical principles. 

\begin{theorem}\label{IPR in Godel doctrine}
Every Gödel first-order hyperdoctrine $\hyperdoctrine{\mathcal{C}}{P}$ satisfies the \bemph{Rule of Independence of Premise}, i.e. whenever $\beta \in P(A \times B)$ and $\alpha\in P(A)$ is a existential-free predicate, it is the case that: \[a:A\; | \; \top \vdash \alpha(a) \rightarrow \exists b.\beta(a,b) \mbox{ implies that } a:A\; | \; \top \vdash \exists b.(\alpha(a)\rightarrow \beta(a,b)).\]

\end{theorem}
\begin{proof}
Let us assume that $a:A\; | \; \top \vdash \alpha(a) \rightarrow \exists b.\beta(a,b)$. Then it is the case that $a : A\;|\; \alpha(a) \vdash \exists b.\beta(a,b)$. Since $\alpha(a)$ is free from the existential quantifier, it is the case that there is a term in context $a:A\; | \;t(a) : B$ such that: \[a:A\; | \; \top \vdash \alpha(a) \rightarrow \beta(a,t(a)).\] Therefore, since: \[ a:A\;|\;\alpha(a) \rightarrow \beta(a,t(a))\vdash \exists b.(\alpha(a) \rightarrow \beta(a,b))\] (as this holds for any predicate $\gamma(a,-)$ in place of the predicate $\alpha_D(a)\rightarrow \beta(a,-)$) we conclude that: \[a:A\; | \; \top \vdash \exists b.(\alpha(a)\rightarrow \beta(a,b)).\]
\end{proof}
Similarly we can prove the following result.
\begin{theorem}\label{theorem weak markov rule}
Every G\"odel first-order hyperdoctrine $\hyperdoctrine{\mathcal{C}}{P}$ satisfies the following \bemph{Modified Markov Rule}, i.e. whenever $\beta_D \in P(A)$ is a quantifier-free predicate and $\alpha \in P(A\times B)$ is an existential-free predicate, it is the case that: 
\[a:A\; | \; \top \vdash (\forall b. \alpha(a,b)) \rightarrow \beta_D(a) \mbox{ implies that } a:A\; | \; \top \vdash \exists b.(\alpha(a,b)\rightarrow \beta_D(a)).\]
\end{theorem}
\begin{proof}
Let us assume that $a:A\; | \; \top \vdash (\forall b. \alpha(a,b)) \rightarrow \beta_D(a) $. Then it is the case that $a:A\; | \;  (\forall b. \alpha(a,b)) \vdash \beta_D(a) $. Hence, since $\beta_D$ is quantifier-free and $\alpha$ is existential-free, there exists a term in context  $a:A\; | \;t(a) : B$ such that:
\[a:A\; | \;  \top\vdash \alpha(a,t(a)) \rightarrow \beta_D(a) \]
therefore, since:
\[ a:A\;|\;\alpha(a,t(a)) \rightarrow \beta(a)\vdash \exists b.(\alpha(a,b) \rightarrow \beta_D(a))\]
we can conclude that:
\[ a:A\; | \; \top \vdash \exists b.(\alpha(a,b)\rightarrow \beta_D(a)).\]
\end{proof}
To obtain Markov rule from Theorem \ref{theorem weak markov rule} we have to require the bottom element to be quantifier-free.
\begin{corollary}
Every G\"odel first-order hyperdoctrine  $\hyperdoctrine{\mathcal{C}}{P}$  such that $\bot$ is a quantifier-free predicate satisfies \bemph{Markov Rule}, i.e. for every quantifier-free element $\alpha_D\in P(A\times B)$ it is the case that:
\[b:B\; | \; \top \vdash \neg \forall a. \alpha_D (a,b) \mbox{ implies that } b:B \; |\; \top \vdash \exists a . \neg \alpha_D (a,b).\]
\end{corollary}
\begin{proof}
It follows by Theorem \ref{theorem weak markov rule} just by replacing $\beta_D$ with $\bot$, that is quantifier-free by hypothesis.
\end{proof}

In Theorems \ref{theorem weak markov rule} and \ref{IPR in Godel doctrine} we proved that the universal properties of existential and universal free elements allow us to prove that a G\"odel first-order hyperdoctrine satisfies Modified Markov Rule and the Rule of Independence of Premise.

From a logical perspective, the intuition behind Theorem,\ref{IPR in Godel doctrine}  is that the existential-free elements of a G\"odel first-order hyperdoctrine correspond to formulae satisfying (R-IP). Similarly, we have that the elements of a G\"odel doctrine that are quantifier-free, are exactly those satisfying a (R-MMP) by Theorem \ref{theorem weak markov rule}.

Notice also that applying the definitions of existential-free and universal-free elements, we immediately obtain the following presentation of the \emph{Rule of Choice}, see \cite{maiettipasqualirosolini} (also called \emph{explicit definability} in \cite{Rathjen2019}) and the \emph{Counterexample Property}, previously defined in \cite{trottapasquale2020}.

\begin{corollary}\label{proposition Counterexample Property}
Every G\"odel first-order hyperdoctrine  $\hyperdoctrine{\mathcal{C}}{P}$ such that $\bot$ is a quantifier-free object satisfies the \bemph{Counterexample Property}, that is, whenever: $$a : A \; |\; \forall b. \alpha(a,b)\vdash \bot$$ for some predicate $\alpha(a,b)\in P(A\times B)$, then it is the case that: $$a : A \; |\; \alpha(a,g(a))\vdash \bot$$ for some term in context $a : A \;|\; g(a) : B$.
\end{corollary}

\begin{corollary}\label{proposition rule of choice}
Every G\"odel first-order hyperdoctrine $\hyperdoctrine{\mathcal{C}}{P}$ such that $\top$ is existential-free satisfies the \bemph{Rule of Choice}, that is, whenever: $$a : A \; |\; \top \vdash \exists b. \alpha (a,b)$$ for some existential-free predicate $\alpha \in P(A\times B)$, then it is the case that: $$a : A \; |\; \top \vdash \alpha (a,g(a))$$ for some term in context $a : A \;|\; g(a) : B$.
\end{corollary}

\subsection{Logical principles}

In the previous results, we have seen which \emph{rules} hold in G\"odel first-order hyperdoctrines. This subsection is devoted to the analysis of the respective logical \emph{principles} in G\"odel first-order hyperdoctrines. In detail, we look for the right hypotheses that allow us to produce models of the stronger formulation of the rules as principles. The following theorem is the first of this series of results and involves the Independence of Premise:

\begin{theorem}\label{IP in Godel doctrine}
Every G\"odel first-order hyperdoctrine $\hyperdoctrine{\mathcal{C}}{P}$ such that existential-free elements are closed with respect to finite conjunctions satisfies the \bemph{Principle of Independence of Premise}, i.e. whenever $\beta \in P(A \times B)$ and $\alpha\in P(A)$ is an existential-free predicate, it is the case that: \[a:A\; | \; \top \vdash ( \alpha(a) \rightarrow \exists b.\beta(a,b) )\rightarrow \exists b.(\alpha(a)\rightarrow \beta(a,b)).\]

\end{theorem}
\begin{proof}
First, since every G\"odel doctrine has enough existential-free elements, there exists an existential-free element $\gamma(a,c)\in P(A\times C)$ such that: $$a:A\; | \; \exists c.\gamma(a,c)\dashv \vdash \alpha(a) \rightarrow \exists b.\beta(a,b).$$
In particular, we have that $a:A,c:C\; | \; \gamma(a,c) \vdash \alpha(a) \rightarrow \exists b.\beta(a,b).$
Then we have that: \[a:A,c:C\; | \; \gamma(a,c) \wedge \alpha(a) \vdash \exists b.\beta(a,b)\]
and $\gamma(a,c) \wedge \alpha(a)$ is an existential-free elements, since both $\gamma(a,c)$ and $\alpha(a)$ are existential-free elements and existential-free elements are closed with respect finite conjunction  by hypethesis. 
Therefore, we can conclude that there exists a term $a:A,c:C\; |\; t(a,c):B$ such that:
\[a:A,c:C\; | \; \gamma(a,c) \wedge \alpha(a) \vdash \beta(a,t(a,c)).\]
Hence we have:
\[a:A,c:C\; | \; \gamma(a,c)  \vdash \alpha(a)\rightarrow \beta(a,t(a,c))\]
and since $\alpha(a)\rightarrow \beta(a,t(a,c)$ is exaclty $(\alpha(a)\rightarrow \beta(a,b))[t(a,c)/b]$ and it always holds that:
\[ a:A,c:C\; | \;  (\alpha(a)\rightarrow \beta(a,b))[t(a,c)/b]\vdash \exists b.(\alpha(a)\rightarrow \beta(a,b))\]
we can conclude that:
\[ a:A,c:C\; | \;  \gamma (a,c)\vdash \exists b.(\alpha(a)\rightarrow \beta(a,b)).\]
Therefore we get that:
\[ a:A\; | \; \exists c. \gamma (a,c)\vdash \exists b.(\alpha(a)\rightarrow \beta(a,b)).\]
and, since $a:A\; | \; \exists c.\gamma(a,c)\dashv \vdash \alpha(a) \rightarrow \exists b.\beta(a,b)$, it is the case that:
 \[a:A\; | \; \top \vdash ( \alpha(a) \rightarrow \exists b.\beta(a,b) )\rightarrow \exists b.(\alpha(a)\rightarrow \beta(a,b)).\]
\end{proof}
As a corollary of the previous result, we obtain the following presentation of the principle (IP*) introduced in Section \ref{subsec: independence of premise} in terms of G\"odel first-order hyperdoctrines. We recall that (IP*) is precisely the form of the Priciple of Independece of Premise we need in the Dialectica interpretation.

\begin{corollary}
Every G\"odel first-order hyperdoctrine $\hyperdoctrine{\mathcal{C}}{P}$ such that the existential-free elements are closed with respect to finite conjunction satisfies (IP*), i.e. whenever $\beta \in P(C \times B)$ and $\alpha_D\in P(A)$ is an quantifier-free predicate, it is the case that: \[-| \; \top \vdash ( \forall a. \alpha_D(a) \rightarrow \exists b.\forall c.\beta(c,b) )\rightarrow \exists b.(\forall a.\alpha_D(a)\rightarrow \forall c.\beta(c,b)).\]
\end{corollary}
\begin{proof}
It follows from Theorem \ref{IP in Godel doctrine} and from the fact that if $\alpha_D$ is quantifier-free then $\forall a. \alpha_D$ is existential-free.
\end{proof}
Similarly, we can prove the following result for Markov principle.

\begin{theorem}\label{theorem generalised markov principle}
Every G\"odel first-order hyperdoctrine $\hyperdoctrine{\mathcal{C}}{P}$ such that existential-free elements are closed with respect to implication satisfies the following \bemph{Modified Markov principle}, i.e. whenever $\beta_D \in P(A)$ is a quantifier-free predicate and $\alpha \in P(A\times B)$ is an existential-free predicate, it is the case that: 
\[a:A\; | \; \top \vdash (\forall b. \alpha(a,b) \rightarrow \beta_D(a)) \rightarrow \exists b.(\alpha(a,b)\rightarrow \beta_D(a)).\]
\end{theorem}
\begin{proof}
Since $\alpha$ is an existential-free predicate and $\beta_D$ is quantifier-free we have that $\forall b. \alpha(a,b) \rightarrow \beta_D(a)$ is an element of $\exfreedoctrine{P}(A)$. Thus, since $\exfreedoctrine{P}$ has enough quantifier-free elements by defintion of G\"odel doctrine, there exists an universal-free predicate of $\exfreedoctrine{P}$, i.e. a quantifier-free predicate $\sigma_D\in \exfreedoctrine{P}(\times C) $ such that $a:A\; | \; \forall c. \sigma_D(a,c) \dashv \vdash \forall b. \alpha(a,b) \rightarrow \beta_D(a)$. In particular, we have that $a:A\; | \; \forall c.\sigma_D(a,c) \wedge \forall b. \alpha(a,b) \vdash  \beta_D(a)$, and hence $a:A\; | \; \forall c.\forall b.(\sigma_D(a,c) \wedge  \alpha(a,b)) \vdash  \beta_D(a)$. Now, since $\beta_D$ is quantifier-free, i.e. it is universal-free in $\exfreedoctrine{P}$, there exist two terms $a:A \;|\; t(a): B$ and $a:A\;|\; t'(a):C$ such that:
\[ a:A\; | \; \sigma_D(a,t'(a)) \wedge  \alpha(a,t(a)) \vdash  \beta_D(a).\]
Therefore we have that $ a:A\; | \; \sigma_D(a,t'(a))  \vdash   (\alpha(a,b) \rightarrow \beta_D(a))[t(a)/b]$. Now, since we always have that  $ a:A\; | \; \forall c. \sigma_D(a,c) \vdash \sigma_D(a,t'(a))$ and $ a:A\; | \;   (\alpha(a,b) \rightarrow \beta_D(a))[t(a)/b] \vdash  \exists b.(\alpha(a,b)\rightarrow \beta_D(a))$, we can conclude that:
\[a:A\; | \; \top \vdash (\forall b. \alpha(a,b) \rightarrow \beta_D(a)) \rightarrow \exists b.(\alpha(a,b)\rightarrow \beta_D(a)).\]
\end{proof}

To obtain the usual presentation of Markov Principle as corollary of Theorem \ref{theorem generalised markov principle}, we simply have to require the bottom element $\bot$ of a G\"odel first-order hyperdoctrine to be \emph{quantifier-free}.
\begin{corollary}\label{theorem Markov}

Every G\"odel first-order hyperdoctrine $\hyperdoctrine{\mathcal{C}}{P}$  such that existential-free elements are closed with respect to implication and $\bot$ is a quantifier-free predicate satisfies \bemph{Markov Principle}, i.e. for every quantifier-free element $\alpha_D\in P(A\times B)$ it is the case that:
\[b:B\; | \; \top \vdash \neg \forall a. \alpha_D (a,b) \rightarrow \exists a . \neg \alpha_D (a,b).\]
\end{corollary}
\begin{proof}
It follows by Theorem \ref{theorem generalised markov principle} just by replacing $\beta_D$ with $\bot$, that is quantifier-free by hypothesis.
\end{proof}
We have proved that under suitable hypotheses, a G\"odel first-order hyperdoctrine satisfies (IP), (MP), (GMP) and the principle of Skolemisation. 

Therefore, combining Theorem \ref{IP in Godel doctrine}, Theorem  \ref{theorem generalised markov principle} (and Corollary \ref{theorem Markov}), and Proposition \ref{proposition skolemisation}, we can repeat the chain of equivalences we provided in Section \ref{section introducing dialectica, MP and IP}, and obtain the following main result.

\begin{theorem}\label{thm (A-->B)^D iff A^D-->B^D}
Let $\hyperdoctrine{\mathcal{C}}{P}$ be a G\"odel first-order hyperdoctrine such that:
\begin{itemize}
    \item existential-free elements are closed with respect to implication and finite conjunction;
    \item falsehood $\bot$ is a quantifier-free predicate.
\end{itemize}

Then for every $\psi_D$ in $P(I\times U\times X )$ and $\phi_D$ in $P(I\times V\times Y)$ quantifier-free predicates of $P$ we have that the formula:
\[ i:I\; |\; \exists u.\forall x . \psi_D (i,u,x)\to\exists v.\forall y.  \phi_D (i,v,y)\]
is provably equivalent to:
\[  i:I\; |\;  \exists f_0,f_1.\forall u,y. (\psi_D(i,u,f_1(i,u,y))\rightarrow \phi_D (i,f_0(i,u),y)).\]
\end{theorem}

Theorem \ref{thm (A-->B)^D iff A^D-->B^D} fully represents a categorical version of the translation of the implication connective in the Dialectica interpretation. In particular, it shows that the equivalence $(\psi\rightarrow \phi)^D\leftrightarrow (\psi^D\rightarrow \phi^D)$ presented in Section \ref{section introducing dialectica, MP and IP} is perfectly modelled by a G\"odel first-order hyperdoctrine satisfying the natural additional closure properties of Theorem \ref{thm (A-->B)^D iff A^D-->B^D}.
%

\begin{remark}
Observe that Theorem \ref{thm (A-->B)^D iff A^D-->B^D} can be considered a stronger version of Theorem \ref{theorem principal 1}. Hence, once more, it converts the rule stated in the latter theorem into an actual principle.

In detail, by the thesis of Theorem \ref{thm (A-->B)^D iff A^D-->B^D}, it is enough to observe that the first sequent of the statement of Theorem \ref{theorem principal 1} is equivalent to the sequent: \[ i:I\; |\; \top\vdash\exists u.\forall x . \psi_D (i,u,x)\to\exists v.\forall y.  \phi_D (i,v,y)\] by the elimination and introduction rules for the implication, and that the second one is equivalent to the following: \[  i:I\; |\;  \top \vdash \exists f_0,f_1.\forall u,y. (\psi_D(i,u,f_1(i,u,y))\rightarrow \phi_D (i,f_0(i,u),y)).\] For this second equivalence one applies the implicational elimination and introduction to convert the second sequent of \ref{theorem principal 1} into: \[  i:I,u:U,y:Y\; |\;  \top \vdash \psi_D(i,u,f_1(i,u,y))\rightarrow \phi_D (i,f_0(i,u),y)\] which is actually equivalent to $i:I\; |\;  \top \vdash \exists f_0,f_1.\forall u,y. (\psi_D(i,u,f_1(i,u,y))\rightarrow \phi_D (i,f_0(i,u),y))$, by Corollary \ref{proposition rule of choice} and being the formula: $$i:I\; |\; \forall u,y. (\psi_D(i,u,f_1(i,u,y))\rightarrow \phi_D (i,f_0(i,u),y))$$ existential-free.

\end{remark}

Theorem \ref{thm (A-->B)^D iff A^D-->B^D} follows as a consequence of the fragment of first-order logic under which the internal language of a G\"odel first-order hyperdoctrine is closed. Observe that this fragment contains at least the whole intuitionistic first-order logic together with the Principle of Independence of Premise, the Modified Markov Principle and the Principle of Skolemisation. These principles, together with the rules of intuitionistic first-order logic are precisely what is needed to get the equivalence $(\psi\rightarrow \phi)^D\leftrightarrow (\psi^D\rightarrow \phi^D)$ in a G\"odel first-order hyperdoctrine.

Clearly any boolean doctrine satisfies these principles as well, as it models every inference rule of classic first-order logic. However, in general they are not satisfied by a usual hyperdoctrine, because they are not necessarily true in intuitionistic first-order logic. It turns out that \textit{the fragment of first-order logic modelled by a G\"odel hyperdoctrine is right in-between the intuitionistic first-order logic and the classical first-order logic}: it is  powerful enough to guarantee the equivalences in Section \ref{section introducing dialectica, MP and IP} that justify the Dialectica interpretation of the implication.

\section{Tripos-to-topos and G\"odel doctrines}
The tripos-to-topos construction was originally introduced in \cite{pitts02,hyland89} as a generalisation of the construction of the category of sheaves of a locale. Recently, this construction has been proved to be an instance of the exact completion of an elementary existential doctrine, we refer to \cite{maiettirosolini13b,maiettipasqualirosolini} for all the details. Let us briefly recall it.

\medskip

\noindent
\textbf{Tripos-to-topos.} Given a first-order hyperdoctrine $\hyperdoctrine{\mC}{P}$, the category $\mathsf{T}_P$ consists of:
 
\begin{itemize}

\item \textbf{objects:} are pairs $(A,\rho)$ where $\rho \in P(A\times A)$  satisfies:

\begin{itemize}

\item \emph{symmetry:} $ a_1,a_2:A\; | \;\rho(a_1,a_2)\vdash \rho(a_2,a_1)$;
\item \emph{transitivity:} $a_1,a_2,a_3:A\; | \;\rho(a_1,a_2)\wedge \rho(a_2,a_3)\vdash\rho(a_1,a_3)$.
\end{itemize}

\item \textbf{arrows} $\arrow{(A,\rho)}{\phi}{(B,\sigma)}$: are objects $\phi\in P(A\times B)$ such that:
\begin{enumerate}
\item $a:A,b:B\; | \;\phi(a,b)\wedge \rho(a,a) \vdash \sigma(b,b)$;
\item $a_1,a_2:A,b:B\; | \;\rho(a_1,a_2)\wedge \phi(a_1,b)\vdash \phi(a_2,b)$; 
\item $a:A,b_1,b_2:B\; | \;\sigma(b_1,b_2)\wedge \phi(a,b_1)\vdash \phi(a,b_2)$;
\item $a:A,b_1,b_2:B\; | \;\phi(a,b_1)\wedge \phi(a,b_2)\vdash \sigma(b_1,b_2)$;
\item $a :A\;|\;\rho(a,a)\vdash \exists b.\phi(a,b)$.
\end{enumerate}

\end{itemize}
Then the following holds:
\begin{theorem}
Let $\hyperdoctrine{\mC}{P}$ be a hyperdoctrine. Then $\mathsf{T}_P$ is an exact category.
\end{theorem}
The construction of the category $\mathsf{T}_P$  can be presented in the more general context of elementary and existential doctrines, and it is also  called the \textbf{exact completion} of the elementary existential doctrine $P$, since it lifts to an adjunction between the 2-category of exact categories and that of elementary and existential doctrine. We refer to  \cite[Cor. 3.4]{maiettirosolini13b} for a complete description of the construction in the general case.

\subsection{Tripos-to-topos and exact completions}
We recall from \cite{trottamaietti2020} a useful characterisation of the tripos-to-topos construction of a first-order hyperdoctrine arising as an existential completion. Again, in the present work we present the results for hyperdoctrines, but the characterisation presented in \cite{trottamaietti2020} works for an arbitrary elementary and existential doctrine. 

To properly present such a characterisation we first need to recall from \cite{maiettipasqualirosolini,maiettirosolini13b,maiettirosolini13a} the construction of the category of \emph{predicates} of a first-order hyperdoctrine. The construction of this category is related to the \emph{comprehension and comprehensive diagonal completion}.

\begin{definition}\label{def comprehension comp}
Given a first-order hyperdoctrine $\hyperdoctrine{\mC}{P}$ we define the \textbf{comprehension completion} $\hyperdoctrine{\mG_P}{P_c}$ of $P$ as follows:
\begin{itemize}
\item an object of $\mG_P$ is a pair $(A,\alpha)$ where $A$ is a set and $\alpha\in P(A)$;
\item an arrow $\arrow{(A,\alpha)}{f}{(B,\beta)}$ is an arrow $\arrow{A}{f}{B}$ such that: $$a:A\;|\;\alpha(a)\vdash \beta(f(a)).$$
\end{itemize}
The fibres $P_c(A,\alpha)$ are given by those elements $\gamma$ of $P(A)$ such that $a:A\;|\;\gamma(a)\vdash \alpha(a)$ (i.e. $\gamma\leq \alpha$). Moreover, the action of $P_c$ on a morphism $f: (B,\beta)\ \rightarrow \ (A,\alpha)$ is defined as $P_c(f)(\gamma)=P_f (\gamma)\, \wedge\, \beta $ i.e.  the predicate: $$b:B  \;|\;\gamma(f(b))\wedge \beta(b)$$
where $\gamma\in P(A)$ is such that $\gamma\leq \alpha$.
\end{definition}

Similarly, the construction which freely adds comprehensive diagonal is provided by the \emph{extensional reflection}. We denote $\delta_A:=\exists_{\Delta} (\top_A)$. According to the internal language of a given doctrine $P$, the element $\delta_A \in P(A \times A)$ corresponds to the predicate: $$a_1\colon A,a_2 \colon A\;|\;a_1=a_2.$$

\begin{definition}\label{def extentiona reflection}
Given an elementary doctrine $\hyperdoctrine{\mC}{P}$ we can define \textbf{extensional reflection} $\doctrine{\mX_P}{P_x}$ of $P$ as follows: the base category $\mX_P$ is the quotient category of $\mC$ with respect to the equivalence relation where $f\sim g$ when: $$\top\vdash f(a)=g(a) \textnormal{  \, \, ( i.e. } \top_A\leq P_{\angbr{f}{g}}(\delta_B) \textnormal{ in category-theoretic notation )}$$
in context $a \colon A$, for two parallel arrows $f,g \colon A \to B$. The equivalence class of a morphism $f$ of $\mC$, i.e. an arrow of $\mX_P$, is denoted by $[f]$.
\end{definition}
Finally, we denoted by $\Pred{P}$ the \textbf{category of predicates} of a doctrine $P$, i.e. the category defined as: $$\Pred{P}:=\mX_{P_c}$$ where $P_c$ is the comprehension completion of $P$. Again, we refer to \cite{maiettipasqualirosolini,maiettirosolini13b,maiettirosolini13a} for a complete description of these constructions. Now we have all the instruments to recall the characterisation of tripos-to-topos of existential completions from \cite{trottamaietti2020}. Such a characterisation essentially shows that every tripos-to-topos of an existential completion is an instance of $(-)_{\mathsf{ex}/\mathsf{lex}}$ completion, namely the exact completion of a lex category in the sense of \cite{CARBONI199879,carboni_magno_1982}.

\begin{theorem}\label{thm: char milly trotta}

Let $\hyperdoctrine{\mC}{P}$ be a first-order hyperdoctrine. Then we have the equivalence:
\[\mathsf{T}_{P^{\exists}}\equiv \Pred{P}_{\mathsf{ex}/\mathsf{lex}}\]
of exact categories.
\end{theorem}

\subsection{Tripos-to-topos for G\"odel first-order hyperdoctrines}

We recall that a G\"odel first-order hyperdoctrine is in particular the existential completion of its subdoctrine of existential-free elements by Theorem \ref{char} and Proposition \ref{ex-un-char}. Therefore, we are able (i.e. have the necessary conditions for) to apply Theorem \ref{thm: char milly trotta} and can  conclude the following characterisation of the tripos-to-topos construction of G\"odel first-order hyperdoctrines:

\begin{theorem}\label{thm: char tripos-to-topos godel hyperdoctrines}
Let $\hyperdoctrine{\mC}{P}$ be a G\"odel first-order hyperdoctrine. The the equivalence of exact categories:
\[\mathsf{T}_{P}\equiv \Pred{\exfreedoctrine{P}}_{\mathsf{ex}/\mathsf{lex}}\]
holds.
\end{theorem}
We recall from \cite{pitts02,hyland89} that when a doctrine is a tripos, then its tripos-to-topos is a topos. Hence, we have the following corollary for G\"odel first-order hyperdoctrines:
\begin{corollary}
Let $\hyperdoctrine{\mC}{P}$ be a G\"odel first-order hyperdoctrine. If $P$ is a tripos, then $ \Pred{\exfreedoctrine{P}}_{\mathsf{ex}/\mathsf{lex}}$ is a topos.
\end{corollary}
Given this corollary we might be tempted to call these toposes Dialectica toposes. These are however different from Biering's Dialectica toposes.

\begin{remark}
The notion of \emph{Dialectica topos}  was introduced in \cite{Biering_dialecticainterpretations} as the tripos-to-topos of a suitable tripos called \emph{dialectica tripos}. In \cite{hofstra_2006} Hofstra characterises triposes arising in terms of ordered PCAs equipped with a filter. This characterisation includes Effective Topos-like triposes,
but also the triposes for relative, modified and extensional realisability and the
dialectica tripos. Therefore, the dialectica tripos can been seen as a tripos araising from a suitable ordered PCA. Triposes given by PCAs are known to be instances of a general completion that freely adds left adjoints along arbitrary maps. Hofstra was the first to obserce this fact, see \cite{hofstra_2006,hofstra_2003}, but later in \cite{Frey2020} and \cite{trottamaietti2020} it is proved that the construction identified by Hofstra is a particular case of the full existential completion of a primary doctrine. 

In this paper we have proposed a different approach to the definition of doctrines related to dialectica interpretation, focusing on the logical principles and rules we need to properly translate the implication connective as in the Dialectica. Therefore, we could say that our approach is more syntactic, and not related to realisability in general. The Dialectica tripos introduced in \cite{Biering_dialecticainterpretations} is not a G\"odel doctrine in general, since G\"odel doctrines are given by an existential completion just along projections (and by a universal completion), while the Dialectica tripos is an instance of the full existential completion. Therefore the Dialectica tripos satisfies different structural properties with respect to an arbitrary G\"odel doctrine. For example, the Dialectica tripos has left adjoints along every arrow, satisfying Beck Chevalley conditions, while in an arbitrary G\"odel doctrine the Beck Chevalley conditions are not satisfied along arbitrary maps. Employing the universal properties of the existential and universal completions, one can show that the Dialectica tripos just contains a G\"odel doctrine, but it is not equivalent to such a doctrine. 
\end{remark}

We can also relate this work to Maietti's work on Joyal's arithmetic universes.
\begin{remark}
Observe that categories arising as tripos-to-topos results of G\"odel first-order hyperdoctrines have the same abstract presentation as the so called \emph{Joyal-arithmetic universes} introduced by Maietti in \cite{maietti2010}. Recall that a Joyal-arithmetic universe is defined as the exact completion $\Pred{\mathcal{S}}_{\mathsf{ex}/\mathsf{lex}}$ of the category of predicates of a Skolem theory $\mathcal{S}$ as defined in \cite[Def. 2.2]{maietti2010}, namely a cartesian category with a parameterised natural
numbers object where all the objects are finite products of the natural numbers
object. 
\end{remark}

\section{Conclusion}
This article is the culmination of various intertwined investigations begun in \cite{trotta_et_al:LIPIcs.MFCS.2021.87} and \cite{trotta-lfcs2022}. 
Inspired by Hofstra \cite{hofstra2011} and Hyland \cite{Hyland2002}, as well as by the work of Maietti and Trotta \cite{trottamaietti2020},
itself inspired by Trotta \cite{trotta2020}, we embarked in the programme of expanding the characterisation of the categorical version of the Dialectica Interpretation, to complete the work in Hofstra as far as the characterisation of the Dialectica is concerned
and to make sure that all the logical principles involved in the interpretation are represented in the 
categorical models obtained. 

Since the work on the fibrational setting seemed too abstract and hard to grasp, especially for the logic audience
we intended to communicate with, we opted for descriptions on 
the level of hyperdoctrines in \cite{trotta-lfcs2022} and 
\cite{dialecticaprinciples2022}, where the doctrines are the poset reflections of the fibrations 
used early on. This crystallised our understanding of the issue of 
quantifier-free formulae in the categorical setting, but also made  clear the import of non-intuitionistic principles such as Independence of Premise and the Markov Principle, that had been discussed by logicians, but not in categorical terms, as far as we are aware. Our investigation is, so far, restricted to the environment of the Dialectica  interpretation, but it has wider reach, helping to complete the program of categorification of logic, as originally suggested by Lawvere.

We hope to carry on exploring other side issues of this investigation. We started connecting this work to the work on categorical realisability and computability, as described by Pitt's tripos theory and the tripos to topos construction~\cite{pitts02} in the final section of this article, but much remains to be done.
A different direction that we have not even started to explore is the extension of our work to generalised versions of the Dialectica interpretation, as already hinted in the text, to dependent type theory \cite{mossvonglehn2018}. 
Finally the work in the original Dialectica category model \cite{depaiva1991dialectica} has had several applications to computer science problems like concurrency theory, in the shape of Petri Nets~\cite{dialecticaPN} and others~\cite{Winskel2022}.
We plan to investigate if these and other applications can be improved by our doctrinal version of the models.
\bibliographystyle{plain}
\bibliography{references}

\begin{thebibliography}{10}

\bibitem{AvigadFeferma70}
J.~Avigad and S.~Feferman.
\newblock Gödel's functional ({D}ialectica) interpretation.
\newblock {\em Handbook of Proof Theory}, 137, 02 1970.

\bibitem{Biering_dialecticainterpretations}
B.~Biering.
\newblock Dialectica interpretations -- a categorical analysis ({PhD Thesis}).
\newblock 2008.

\bibitem{carboni_magno_1982}
A.~Carboni and R.~Celia Magno.
\newblock The free exact category on a left exact one.
\newblock {\em Journal of the Australian Mathematical Society. Series A. Pure
  Mathematics and Statistics}, 33(3):295–301, 1982.

\bibitem{CARBONI199879}
A.~Carboni and E.V. Vitale.
\newblock Regular and exact completions.
\newblock {\em Journal of Pure and Applied Algebra}, 125(1):79 -- 116, 1998.

\bibitem{dePaiva1989dialectica}
V.~de~Paiva.
\newblock The {D}ialectica categories.
\newblock {\em Categories in Computer Science and Logic}, 92:47--62, 1989.

\bibitem{depaiva1991dialectica}
V.~de~Paiva.
\newblock The {D}ialectica categories, phd thesis.
\newblock Technical report, University of Cambridge, Computer Laboratory, 1991.

\bibitem{dialecticaPN}
E.~di~Lavore, W.~Leal, and V.~de~Paiva.
\newblock Dialectica petri nets.
\newblock {\em arXiv, 2105.12801}, 2021.

\bibitem{Frey2020}
J.~Frey.
\newblock Categories of partial equivalence relations as localizations.
\newblock {\em preprint}, 2020.

\bibitem{girard1987}
J.Y. Girard.
\newblock Linear logic.
\newblock {\em Theoretical computer science}, 50(1):1--101, 1987.

\bibitem{goedel1986}
K.~G\"odel, S.~Feferman, et~al.
\newblock {\em Kurt G{\"o}del: Collected Works: Volume II: Publications
  1938-1974}, volume~2.
\newblock Oxford University Press, 1986.

\bibitem{godel58}
K.~Gödel.
\newblock Über eine bisher noch nicht ben\"utzte erweiterung des finiten
  standpunktes.
\newblock {\em Dialectica}, 12(3-4):280--287, 1958.

\bibitem{hofstra_2003}
P.~Hofstra.
\newblock {\em Completions in Realizability}.
\newblock PhD thesis, Universiteit Utrecht, 2003.

\bibitem{hofstra_2006}
P.~Hofstra.
\newblock All realizability is relative.
\newblock {\em Mathematical Proceedings of the Cambridge Philosophical
  Society}, 141(2):239–264, 2006.

\bibitem{hofstra2011}
P.~Hofstra.
\newblock The {D}ialectica monad and its cousins.
\newblock {\em Models, logics, and higherdimensional categories: A tribute to
  the work of {M}ih{\'a}ly {M}akkai}, 53:107--139, 2011.

\bibitem{Hyland2002}
J.M. Hyland.
\newblock Proof theory in the abstract.
\newblock {\em Annals of Pure and Applied Logic}, 114(1):43 -- 78, 2002.
\newblock Troelstra Festschrift.

\bibitem{hyland89}
J.M. Hyland and A.M. Johnstone, P.T.and~Pitts.
\newblock Tripos theory.
\newblock {\em Math. Proc. Camb. Phil. Soc.}, 88:205--232, 1980.

\bibitem{Kohlenbach2008}
U.~Kohlenbach.
\newblock Gödel's functional interpretation and its use in current
  mathematics.
\newblock {\em Dialectica}, 62(2):223--267, 2008.

\bibitem{lawvere1969}
F.W. Lawvere.
\newblock Adjointness in foundations.
\newblock {\em Dialectica}, 23:281--296, 1969.

\bibitem{lawvere1969b}
F.W. Lawvere.
\newblock Diagonal arguments and cartesian closed categories.
\newblock In {\em Category Theory, Homology Theory and their Applications},
  volume~2, page 134–145. Springer, 1969.

\bibitem{lawvere1970}
F.W. Lawvere.
\newblock Equality in hyperdoctrines and comprehension schema as an adjoint
  functor.
\newblock In A.~Heller, editor, {\em New York Symposium on Application of
  Categorical Algebra}, volume~2, page 1–14. American Mathematical Society,
  1970.

\bibitem{maietti2010}
M.E. Maietti.
\newblock Joyal's arithmetic universe as list-arithmetic pretopos.
\newblock {\em Theory and Applications of Categories}, 24(3):39--83, 2010.

\bibitem{maiettipasqualirosolini}
M.E. Maietti, F.~Pasquali, and G.~Rosolini.
\newblock Triposes, exact completions, and {H}ilbert's $\varepsilon$-operator.
\newblock {\em Tbilisi Mathematical Journal}, 10, 11 2017.

\bibitem{maiettirosolini13a}
M.E. Maietti and G.~Rosolini.
\newblock Quotient completion for the foundation of constructive mathematics.
\newblock {\em Log. Univers.}, 7(3):371--402, 2013.

\bibitem{maiettirosolini13b}
M.E. Maietti and G.~Rosolini.
\newblock Unifying exact completions.
\newblock {\em Appl. Categ. Structures}, 23:43--52, 2013.

\bibitem{manighetti2016computational}
M.~Manighetti.
\newblock Computational interpretations of {M}arkov's principle, 2016.

\bibitem{mossvonglehn2018}
S.K. Moss and T.~von Glehn.
\newblock Dialectica models of type theory.
\newblock In {\em 33rd Annual ACM/IEEE Symposium on Logic in Computer Science},
  page 739–748, New York, NY, USA, 2018. Association for Computing Machinery.

\bibitem{Rathjen2019}
T.~Nemoto and M.~Rathjen.
\newblock The independence of premise rule in intuitionistic set theories, 11
  2019.

\bibitem{pitts02}
A.~M. Pitts.
\newblock Tripos theory in retrospect.
\newblock {\em Math. Struct. in Comp. Science}, 12:265--279, 2002.

\bibitem{pitts95}
A.M. Pitts.
\newblock Categorical logic.
\newblock In S.~Abramsky, D.~M. Gabbay, and T.~S.~E. Maibaum, editors, {\em
  Handbook of Logic in Computer Science}, volume~6, pages 39--.129. Oxford
  Univ. Press, 1995.

\bibitem{Troelstra1973}
A.S. Troelstra.
\newblock {\em Metamathematical Investigation of Intuitionistic Arithmetic and
  Analysis}.
\newblock Springer, Berlin, 1973.

\bibitem{trotta2020}
D.~Trotta.
\newblock The {E}xistential {C}ompletion.
\newblock {\em Theory and Applications of Categories}, 35:1576--1607, 2020.

\bibitem{trottamaietti2020}
D.~Trotta and M.E. Maietti.
\newblock Generalized existential completions and their regular and exact
  completions.
\newblock {\em arXiv}, 2021.

\bibitem{trottapasquale2020}
D.~Trotta and M.~Spadetto.
\newblock Quantifier completions, choice principles and applications.
\newblock {\em arXiv}, 2020.

\bibitem{trotta_et_al:LIPIcs.MFCS.2021.87}
D.~Trotta, M.~Spadetto, and V.~de~Paiva.
\newblock {The G\"{o}del Fibration}.
\newblock In {\em 46th International Symposium on Mathematical Foundations of
  Computer Science (2021)}, volume 202 of {\em LIPIcs}, pages 87:1--87:16,
  2021.

\bibitem{trotta-lfcs2022}
D.~Trotta, M.~Spadetto, and V.~de~Paiva.
\newblock Dialectica logical principles.
\newblock In S.~Artemov and A.~Nerode, editors, {\em Logical Foundations of
  Computer Science}, pages 346--363, Cham, 2022. Springer International
  Publishing.

\bibitem{dialecticaprinciples2022}
D.~Trotta, M.~Spadetto, and V.~de~Paiva.
\newblock Dialectica logical principles, not only rules.
\newblock {\em Manuscript (submitted)}, 2022.

\bibitem{van_Oosten_realizability}
J.~van Oosten.
\newblock {\em Realizability: An Introduction to its Categorical Side}.
\newblock ISSN. Elsevier Science, 2008.

\bibitem{Winskel2022}
G.~Winskel.
\newblock Making concurrency functional.
\newblock {\em arXiv, 2202.13910}, 2022.

\end{thebibliography}
\end{document}